\documentclass[reqno,12pt]{amsart}

\usepackage{geometry,amsmath,amsfonts,amssymb,amsthm,amscd,mathrsfs,graphicx,enumerate,multicol,wrapfig,subfigure,hyperref,stmaryrd, accents,emptypage}
\usepackage[all]{xy}
\usepackage{setspace}
\usepackage[T1]{fontenc}

\numberwithin{equation}{subsection}

\newcommand{\ra}{\rightarrow}

\newcommand{\lra}{\longrightarrow}
\newcommand{\8}{\infty}
\newcommand{\p}{\prime}
\newcommand{\pt}{\partial}
\newcommand{\eset }{\emptyset}

\newcommand{\al}{\alpha}

\newcommand{\om}{\omega}
\newcommand{\gam}{\gamma}

\newcommand{\vp}{\varphi}
\newcommand{\lam}{\lambda}

\newcommand{\q}{\theta}

\newcommand{\be}{\beta}
\newcommand{\dt}{\delta}

\newcommand{\Zbb}{\mathbb{Z}}
\newcommand{\Pbb}{\mathbb{P}}
\newcommand{\Cbb}{\mathbb{C}}

\theoremstyle{plain} 

\newtheorem{THM}{Theorem}[section]
\newtheorem{DEF}[THM]{Definition}
\newtheorem{EX}[THM]{Example}

\newtheorem{QUE}[THM]{Question}

\newtheorem{ILL}[THM]{Illustration}
\newtheorem{PROP}[THM]{Proposition}
\newtheorem{LEM}[THM]{Lemma}
\newtheorem{COR}[THM]{Corollary}
\newtheorem{REM}[THM]{Remark}

\newtheorem{CONJ}[THM]{Conjecture}

\newcommand{\bt}{\bullet}

\newcommand{\Ac}{\mathcal{A}}
\newcommand{\Bc}{\mathcal{B}}

\newcommand{\img}{\mathrm{im}}

\newcommand{\Vc}{\mathcal{V}}
\newcommand{\Uc}{\mathcal{U}}

\newcommand{\Oc}{\mathcal{O}}
\newcommand{\red}{\mathrm{red}}
\newcommand{\Hom}{\mathrm{Hom}}

\newcommand{\Nc}{\mathcal{N}}

\newcommand{\Gc}{\mathcal{G}}

\newcommand{\Fc}{\mathcal{F}}

\newcommand{\Xfr}{\mathfrak{X}}

\newcommand{\Tc}{\mathcal{T}}

\newcommand{\Ufr}{\mathfrak{U}}

\newcommand{\Scl}{\mathcal{S}}
\newcommand{\Xc}{\mathcal{X}}
\newcommand{\Mcl}{\mathcal{M}}

\newcommand{\Yfr}{\mathfrak{Y}}

\newcommand{\Qcl}{\mathcal{Q}}

\usepackage{xcolor}
\definecolor{airforceblue}{rgb}{0.36, 0.54, 0.66}
\definecolor{burgundy}{rgb}{0.5, 0.0, 0.13}
\definecolor{majorelleblue}{rgb}{0.38, 0.31, 0.86}
\definecolor{darkblue}{rgb}{0.0, 0.0, 0.55}

\hypersetup{urlcolor=burgundy, linkcolor=burgundy,
citecolor=darkblue,
colorlinks=true}

\newcommand{\RNum}[1]{\uppercase\expandafter{\romannumeral #1\relax}}

\usepackage{chngcntr}
\counterwithin{section}{part} 

\title[Fam. Smflds., Splt. Types and Obstr. Maps]
{Families of Supermanifolds: Splitting Types and Obstruction Maps
\\}
\author{\small Kowshik Bettadapura}
\date{}

\begin{document}

\begin{abstract} 
In this article we study the notion of supermanifolds families, starting from Green's general classification of supermanifolds. The topics studied divide this article into two distinct parts, labelled \RNum{1} and \RNum{2} respectively. Part \RNum{1} concerns the splitting type of a supermanifold and our attempt to construct variations thereof. In Part \RNum{2} we look at a particular class of families, referred to as `$gtm$-families'. We are concerned with the notion of a `secondary obstruction theory', which is apparent in any $gtm$-family of supermanifolds over a (connected and Stein) non-reduced, superspace base. 
\end{abstract}

\maketitle

\setcounter{tocdepth}{1}
\tableofcontents

\onehalfspacing

\section*{Introduction}

\noindent
In this article we study families of supermanifolds over (I) a Stein base; and (II) a Stein superspace base. Hence, the article can be divided into two parts with the former, (I), forming the foundations for the latter (II). The questions addressed in each part are quite distinct. In the former, (I), we are primarily concerned with the splitting type. In the latter, (II), we introduce a notion of a `secondary obstruction to splitting' and consider an obstruction problem in analogy with previous work by the author in \cite{BETTPHD, BETTOBSTHICK, BETTAQ}.

We summarise the motivation and content of this article in what follows.

\subsection*{Part I: The Splitting Type}

\subsubsection*{Motivation}
The splitting type of a complex supermanifold serves as a first step toward a general classification. The existence of exotic structures however undermines the power of a classification based on splitting type alone. For more details on exotic structures on supermanifolds, see \cite{DW1, BETTHIGHOBS}. The original motivation behind this article was to study these exotic structures by constructing `variations of splitting type', which we take to be a family wherein nearby fibers are exotic structures on some fixed supermanifold. This leads to a motivating problem:
\begin{align*}
\begin{array}{l}
\mbox{{\bf Problem.}\emph{ Construct an isotrivial family of}}\\
\mbox{\emph{supermanifolds with varying splitting type.}}
\end{array}
\end{align*}
~
\subsubsection*{Brief Summary} The first part of this article catalogues our attempt to construct variations of splitting type. After presenting a number of definitions and general constructions we arrive at our main result which, in a sense, realises the opposite of our initial objective. Our main result is Theorem \ref{rg784gf794hf98hf0j3fj34444} where we find: over a connected, Stein base, the splitting type of the fiber of a supermanifold family will generally \emph{not} vary with respect to parameters on the base. This leads to a proposal that, over a general base, the splitting type might vary semi-continuously, in analogy with a classical result by Kodaira and Spencer on the dimension of the deformation space of a complex manifold\footnote{i.e., the dimension of its first cohomology, valued in its tangent sheaf.}.

\subsubsection*{Outline and Main Results}
The definition of supermanifold families which we propose in Definition \ref{fbvryuvuyviuoijfioejfioe} mirrors the classical definition for complex manifolds given in \cite{KS}. We refer to the notion of embedding detailed by Donagi and Witten in \cite{DW1} and studied further in \cite{BETTEMB}. The rudiments of supermanifolds and embeddings thereof is recalled in Section \ref{fjnvkbvueuvbuenviepve}. Subsequently, in Section \ref{fihrhf948hf84jf0j40fj903} we are concerned with the definition of supermanifold families, the notion of a `splitting type deformation' (Definition \ref{fvuirviurvieiojowjwo}) and some immediate constructions, such as the `classifying diagrams' and `classifying maps'. With these we comment on the splitting problem for the total space of supermanifold families in Proposition \ref{rf4h78fh498fh9f038j} and Corollary \ref{rf847gf784hf9hf983}. In Section \ref{fvirubvuirbiurburnornoir} we look to import the notion of `(local) triviality' for complex analytic families in \cite{KS} to supermanifold families. Generally, we view a supermanifold as being a structure associated to the data of a \emph{model} $(X, T^*_{X, -})$, where $X$ is a complex manifold and $T^*_{X, -}$ is a holomorphic vector bundle. As such, a \emph{family of supermanifolds} over a base $B$ consists of a supermanifold modelled on the data $(Y, T^*_{Y, -})$, with fiber over a point in $B$ being some supermanifold modelled on $(X, T^*_{X,-})$. Hence, in order to define a supermanifold family, we need to realise the model of the total space $(Y, T^*_{Y, -})$ as a `family of models' over $B$ with `typical fiber' the model $(X, T^*_{X, -})$. It is to such families of models $(Y, T^*_{Y, -})$ that we ascribe a kind of `triviality'. We consider three kinds of triviality: \emph{local triviality, isotriviality} and \emph{global triviality}. This section culminates in Definition \ref{rf784gf794hf984hf84hf0} where supermanifold families are defined to be of type $ltm$-, $ism$-, or $gtm$-, depending on whether its total space is modelled on locally trivial, isotrivial or globally trivial model. Subsequent considerations in this article only concern the $gtm$-families. Section \ref{tgh748g749hg84093j93j333} begins by defining the notion of analyticity for supermanifold families in general (Definition \ref{rfh489f98hfj3039k}). The objective in this section is to then study how the splitting type might vary in \emph{analytic, $gtm$-families}. Our investigation culminates in Theorem \ref{rg784gf794hf98hf0j3fj34444} which asserts, in essence, that the splitting type will not vary in $gtm$-families over a connected, Stein base. The proof of Theorem \ref{rg784gf794hf98hf0j3fj34444} involves a useful, tensor product decomposition of the obstruction class to splitting the total space of a $gtm$-family over a Stein base, derived in Proposition \ref{rh973hf983hf80jf093fj}. This is important in forming the notion of a `generically isotrivial family' in Definition \ref{ryurgfyugfiuhfueofioefoeo}. The remainder of this section is then devoted to some applications of Theorem \ref{rg784gf794hf98hf0j3fj34444} and consequences of its proof. Section \ref{rfg78gf78gf9h38fh3f03} concludes the first part of this article. Here we revisit a classical construction by Rothstein in \cite{ROTHDEF} of a one dimensional family of supermanifolds. This is referred to as \emph{Rothstein's deformation}. Our objective is to illustrate many of the concepts introduced so far in this article. We show that Rothstein's deformation is an example of a \emph{$gtm$-family of supermanifolds over a connected Stein base that is generically isotrivial}. We end Section \ref{rfg78gf78gf9h38fh3f03} on Theorem \ref{rf8g874g974h98fh8fj309j} showing how Rothstein deformations can be glued over the projective line, thereby yielding an example of a family of supermanifolds over a non-Stein base.

\subsection*{Part II: Obstruction Maps}

\subsubsection*{Motivation}
A common example running through this article are the deformations of super Riemann surfaces, studied in \cite{DW1, DW2, BETTPHD, BETTSRS}. They provide the central motivation behind many of the definitions given. As observed in Remark \ref{fkvnkbvbvnvinvoinenvoe} however, the first part of this article only concerns supermanifold famlies over a \emph{reduced} base, whereas deformations of super Riemann surfaces would be supermanifold families over a non-reduced, i.e., superspace base\footnote{in the above-cited articles, it is $\Cbb^{0|n}$.}. It is therefore desirable to extend the theory in the first part to families over superspace bases. In \cite[p. 29]{DW2} deformations of any split model were given by reference to an almost complex structure. Yet, these methods do not generalise to give deformations of non-projected supermanifolds. Thus a motivating problem here is taken to be:
\begin{align*}
\begin{array}{l}
\mbox{{\bf Problem.} \emph{Construct deformations of non-projected}} 
\\
\mbox{\emph{supermanifolds over a superspace base}.}
\end{array}
\end{align*}

\subsubsection*{Brief Summary}
This part of the article is an attempt by the author to give a general notion of supermanifold families over superspace bases. The main result here is in arriving, from our definitions, at an obstruction problem of a similar kind considered in \cite{BETTOBSTHICK, BETTAQ}. We derive a host of `obstruction maps' and conjecture that they can be reduced to an explicit construction involving a certain class associated to the model of the total space. This is confirmed in a particular case.

\subsubsection*{Outline and Main Results}
Section \ref{fjbvhvvyuviruriormprmopvmrp} begins with our definition of a supermanifold family over a superspace base. As in the previous part, we focus on the notion of a `family of models' in Definition \ref{rfh74gf9hf830fj93jf} and use this as a basis for defining a family of supermanifolds in Definiton \ref{ckjvbbvjfbvhjbnckld}. We are only interested in families of $gtm$-type here over a Stein superspace base\footnote{A Stein superspace base is taken to be a superspace whose reduced part is a Stein space.}.  Our definition of a supermanifold family over a superspace base is made so as to build on supermanifold families over reduced bases in the first part. Thus, any $gtm$-family of supermanifolds over a superspace base will contain a $gtm$-family of supermaifolds over a reduced base, referred to as its \emph{underlying family}. This leads naturally to a `projectability problem' for $gtm$-families, which is illustrated by appeal to deformations of super Riemann surfaces in Example \ref{rgf784gf73hf83h0j9f33333}. We turn now to the problem of comparing the obstruction class to splitting the total space to the class splitting the underlying family. This is studied in analogy with a similar problem in the previous part concerning the embedding of the fiber of a supermanifold family, which involved the derivation of a `compatibility diagram' (see, Illustration \ref{h78gf87g9h38hf03f093j}). The compatibility diagram in the present case is more complicated however since the model of a $gtm$-family, $(Y, T^*_{Y, -})$, can itself be `non-trivial'. This is in the following sense: the odd cotangent bundle $T^*_{Y, -}$ of the model is taken to be an extension of sheaves and its extension class, defined to be the model class, need not be trivial. In particular, the exterior powers of $T^*_{Y, -}$ will admit a filtration and therefore so will the obstruction space of the $gt$-model. In Proposition \ref{fjdkbvjvbebvuibvevbeo} we see how the filtration on the exterior powers $\wedge^{j^\p}T^*_{Y -}$ are related to the compatibility diagram coming from the embedding of the underlying family into the total space. Section \ref{fjvbgbvirbuiburnvoenoinveo} is devoted to unpacking Proposition \ref{fjdkbvjvbebvuibvevbeo} and further studying this filtration. We begin with the observation: If the obstruction to splitting the underlying family vanishes, then there will exist a `seconadary obstruction' to splitting the total space of the $gtm$-family. Hence we are led to a refined notion of `obstruction theory for $gtm$-families' in analogy with work in \cite{BETTOBSTHICK}, established more precisely in Question \ref{rfh849hf84jf0jf933333}. We derive a bounded, differential complex in Proposition \ref{rg78gf87gf973h98hf08h3} whose differentials measure the failure to resolve Question \ref{rfh849hf84jf0jf933333} in the affirmative. This is the content of Proposition \ref{rfh74fg874f93hf83j0f93}. In effort to further understand the differentials in this complex, we present a parallel construction of maps which reference the model class of the $gt$-family explicitly. In Conjecture \ref{rhf74gf74gf983hf80309fj30} we propose a relation between these latter, explicit constructions and the differentials. We conclude on Proposition \ref{fjbbvubiuvbuenoienienp} which is an affirmation of Conjecture \ref{rhf74gf74gf983hf80309fj30} in a particular instance. Its proof is deferred to Appendix \ref{rfh97hfg9hf83jf09jf9j333}. We conclude this article with Section \ref{rygf8g478gf794hf893h893h03333333} where a brief discussion of obstruction complexes is given. We present some general questions and conjectures pertaining to the generalisation of the compatibility diagram between obstruction spaces to obstruction complexes.
\\\\
{\bf Remark.}
It is important to mention that, in the algebraic setting, Vaintrob in \cite{VAIN} considered the problem of deformations of supermanifolds in considerable generality. There is however little overlap between the present work and Vaintrob's study since, in the latter, it is unclear as to what role is played by splitting types and secondary obstructions.
\\\\
\emph{Disclaimer.} The spaces and maps considered in this article are assumed to be smooth and defined over the complex numbers; and the abelian sheaves are all locally free.

\part{Splitting Types}

\section{Preliminaries: Supermanifolds and Embeddings}
\label{fjnvkbvueuvbuenviepve}

\subsection{Supermanifolds}
The notion of supermanifold to be understood in this article follows that in \cite{BETTEMB, BETTHIGHOBS, BETTAQ}. We refer to these articles for further details and only sketch the rudiments here. Let $(Y, T^*_{Y, -})$ be a pair comprising a complex manifold $Y$ and holomorphic vector bundle $T^*_{Y, -}$. The manifold $Y$ is referred to as the \emph{reduced space} and $T^*_{Y, -}$ as either the \emph{odd cotangent bundle} or \emph{odd conormal bundle}. 

\begin{DEF}
\emph{The \emph{dimension} of a model $(Y, T^*_{Y, -})$ is taken to be the pair $(p^\p|q^\p)$ where $p^\p = \dim Y$ and $q^\p = \mathrm{rank}~T^*_{Y, -}$. If we wish to either $p^\p$ or $q^\p$ then we will write $\dim_+(Y, T^*_{Y, -}) = p^\p$ and $\dim_-(Y, T^*_{Y, -}) = q^\p$.}
\end{DEF}

\noindent
Green in \cite{GREEN} gave a general classification of supermanifolds modelled on $(Y, T^*_{Y, -})$ via the \v Cech cohomology set $\mbox{\v H}^1\big(Y, \mathcal G_{T^*_{Y, -}}^{(2)}\big)$, where $\Gc_{T^*_{Y, -}}^{(2)}$ is a filtered sheaf of groups with filtered pieces $\{\Gc^{(j^\p)}_{T^*_{Y, -}}\subset \Gc_{T^*_{Y, -}}^{(2)}\}_{j^\p = 3, \ldots, \dim_-(Y, T^*_{Y, -})}$. Berezin in \cite{BER} and, more recently, Donagi and Witten in \cite{DW1} observed that the \v Cech cohomology of the $j^\p$-th filtered piece $\mbox{\v H}^1\big(Y, \Gc^{(j^\p)}_{T^*_{Y, -}}\big)$ classifies `supermanifolds with level $j^\p$-structure'. Our observation is that the object $\mbox{\v H}^1\big(Y, \Gc^{(j^\p)}_{T^*_{Y, -}}\big)$ can be formed by reference to the model $(Y, T^*_{Y, -})$ alone and does not require knowledge of the conventional definition of supermanifolds as locally ringed spaces, which can be found e.g., in \cite{BER, YMAN, QFAS}. Given a model $(Y, T^*_{Y, -})$ then, we are contented with the following definition of a supermanifold \emph{with prescribed splitting type}.

\begin{DEF}\label{furbfbiuff3f03jf3jf}
\emph{A supermanifold modelled on $(Y,T^*_{Y, -})$ with splitting type $j^\p>1$ is a representative of a class in $\mbox{\v H}^1\big(Y, \Gc^{(j^\p)}_{T^*_{Y, -}}\big)$.}
\end{DEF}

\noindent
The set $\mbox{\v H}^1\big(Y, \Gc^{(j^\p)}_{T^*_{Y, -}}\big)$ is a pointed set for each $j^\p$ and the inclusion $\Gc^{(j^\p)}_{T^*_{Y, -}}\subset \Gc^{(j^{\p\p})}_{T^*_{Y, -}}$ induces a map of pointed sets $\mbox{\v H}^1\big(Y, \Gc^{(j^\p)}_{T^*_{Y, -}}\big) \ra \mbox{\v H}^1\big(Y, \Gc^{(j^{\p\p})}_{T^*_{Y, -}}\big)$ for any $j^\p> j^{\p\p}$. In particular, the base point is preserved. Now $\Gc_{T^*_{Y, -}}^{(N)} = (1)$ if $N> \dim_- T^*_{Y, -}$ so that $\mbox{\v H}^1\big(Y, \Gc^{(N)}_{T^*_{Y, -}}\big)$ will be a singleton set. Its image in $\mbox{\v H}^1\big(Y, \Gc^{(j^\p)}_{T^*_{Y, -}}\big)$ for any $j^\p$ coincides with the base point. We set $\mbox{\v H}^1\big(Y, \Gc^{(\8)}_{T^*_{Y, -}}\big) \stackrel{\Delta}{=}\mbox{\v H}^1\big(Y, \Gc^{(N)}_{T^*_{Y, -}}\big)$ if $N> \dim_-(Y,T^*_{Y, -})$. 

\begin{DEF}
\emph{A representative of the base point in $\mbox{\v H}^1\big(Y, \Gc^{(\8)}_{T^*_{Y, -}}\big)$ will be referred to as the \emph{split model}.}
\end{DEF}

\noindent
Hence if a supermanifold modelled on $(Y, T^*_{Y, -})$ has splitting type $j^\p > \dim_-(Y, T^*_{Y, -})$ it will be the split model and we will say it is a supermanifold with splitting type $j^\p = \8$.

\subsubsection{Notation}
Generally, \v Cech cohomology of a space $Y$ valued in a sheaf $\Tc$ is defined as the direct limit over common refinement of open covers $(\Ufr\ra Y)$,
\[
\mbox{\v H}^p\big(Y, \Tc\big)
\stackrel{\Delta}{=}
\varinjlim_{(\Ufr\ra Y)} \mbox{\v H}^p\big(\Ufr, \Tc\big)
\]
where 
\[
\mbox{\v H}^p\big(\Ufr, \Tc\big) = \frac{\ker\{\dt: C^{p}(\Ufr, \Tc) \ra  C^{p+1}(\Ufr, \Tc)\}}{\mathrm\img\{\dt: C^{p-1}(\Ufr, \Tc) \ra  C^p(\Ufr, \Tc)\}}.
\]
for $C^p(\Ufr, \Tc)$ the group of $\Tc$-valued, $p$-cochains on the cover $\Ufr$. When $p = 0, 1$
these constructions make sense for $\Tc$ a sheaf of not-necessarily-abelian groups. When $p = 1$, a representative of a class in $\mbox{\v H}^p\big(Y, \Tc\big)$ is given by the data $(\Ufr\ra Y, \vartheta)$ where $\Ufr = \big((U_\al)_{\al\in I}, (U_{\al\be})\big)$ is an open cover of $Y$ together with a collection of intersections $U_{\al\be}\subset U_\al$ for each pair of indices $(\al, \be)$; and $\vartheta = (\vartheta_{\al\be})$ are isomorphisms $\vartheta_{\al\be} : \Tc(U_{\al\be}) \stackrel{\cong}{\ra} \Tc(U_{\be\al})$ with inverse $\vartheta_{\be\al}$ and subject to the cocycle condition on all triple intersections, i.e., for all triples of indices $\{\al, \be, \gam\}$ that,
\begin{align}
\vartheta_{\al\gam} = \vartheta_{\be\gam}\circ\vartheta_{\al\be}.
\label{fbvkbvbviuebvuibeuveoeinvie}
\end{align}
A supermanifold modelled on $(Y, T^*_{Y, -})$ of splitting type $j^\p$ is thus a collection of such data $(\Ufr\ra Y,\vartheta)$ with $\Tc = \Gc^{(j^\p)}_{T^*_{Y, -}}$. We will denote this data by $(\Uc\ra\Xc)$. The split model will be denoted $S(Y, T^*_{Y, -})$.

\begin{REM}\label{fknkrnjkvbrvbvnoinio4nfo4f4}
\emph{This notation is suggestive of the idea that a supermanifold is a space $\Xc$ in its own right equipped with a covering $\Uc$, much like $Y$ is a space with a covering $\Ufr$. The reduced space of $(\Uc\ra\Xc)$ is then $(\Ufr\ra Y)$. Abstractly, $\Xc = (Y, \Oc_\Xc)$ is a locally ringed space with structure sheaf $\Oc_\Xc$ locally isomorphic to $\wedge^\bt T^*_{Y, -}$. In this way, we can recover the locally-ringed-space-definition of supermanifolds from Definition \ref{furbfbiuff3f03jf3jf}.}
\end{REM}

\subsection{Obstructions}
Any supermanifold of splitting type $j^\p$ will define an \emph{obstruction to splitting}. To a model $(Y, T^*_{Y, -})$ set,
\[
T^*_{Y, (\pm)^{j^\p}}
=
\left\{
\begin{array}{ll}
T^*_Y & \mbox{if $j^\p$ is even}
\\
T^*_{Y, -}&\mbox{if $j^\p$ is odd}
\end{array}
\right.
\]
The $j^\p$-th obstruction sheaf of the model $(Y, T^*_{Y, -})$ is then:
\begin{align}
\Ac_{T^*_{Y, -}}^{(j^\p)} 
\stackrel{\Delta}{=}
\mathcal Hom_{\Oc_Y}\big(T^*_{Y, (\pm)^{j^\p}}, \wedge^{j^\p}T^*_{Y, -}\big)
\label{fbvkbvyibiufioenievpmpeoe}
\end{align}
As shown by Green in \cite{GREEN}, the $j^\p$-th obstruction sheaf can be recovered as the quotient $\Gc_{T^*_{Y, -}}^{(j^\p)}/\Gc_{T^*_{Y, -}}^{(j^\p+1)}$. Hence there is a map induced on cohomology $\om:\mbox{\v H}^1\big(Y, \Gc^{(j^\p)}_{T^*_{Y, -}}\big)\ra H^1\big(Y, \Ac_{T^*_{Y, -}}^{(j^\p)}\big)$ which sends a supermanifold modelled on $(Y, T^*_{Y, -})$ with splitting type $j^\p$ to a class in the cohomology group $H^1\big(Y, \Ac_{T^*_{Y, -}}^{(j^\p)}\big)$

\begin{DEF}
\emph{If $(\Uc\ra\Xc)$ is a supermanifold modelled on $(Y, T^*_{Y, -})$ of splitting type $j^\p$, then the class  $\om(\Uc\ra\Xc)\in H^1\big(Y, \Ac^{(j^\p)}_{T^*_{Y, -}}\big)$ is referred to as the \emph{obstruction to splitting $(\Uc\ra\Xc)$}.}
\end{DEF}

\noindent
The cohomology group $H^1\big(Y, \Ac_{T^*_{Y, -}}^{(j^\p)}\big)$ is referred to as the \emph{$j^\p$-th obstruction space of the model $(Y, T^*_{Y, -})$}.

\subsection{Embeddings of Supermanifolds}
Following Donagi and Witten's characterisation of embeddings in \cite{DW1}, inspiring subsequent work by the author in \cite{BETTEMB}, we recall: given models $(X, T^*_{X, -})$ and $(Y, T^*_{Y, -})$, the former embeds into the latter iff we have:
\begin{align}
X\stackrel{i}{\subset} Y
&&\mbox{and}&&
i^*T^*_{Y, -} \lra T^*_{X, -} \lra 0.
\label{ffgggruepepedjfhffhfj}
\end{align}
To any embedding of models $(X, T^*_{X, -})\subset (Y, T^*_{Y, -})$ there exist maps on cohomology for each $j$:
\[
\xymatrix{
\mbox{\v H}^1\big(Y, \Gc^{(j)}_{T^*_{X, -}; T^*_{Y,-}}\big)\ar[d]_{r_*} \ar[r]^{u_*} & \mbox{\v H}^1\big(Y, \Gc^{(j)}_{T^*_{Y,-}}\big)
\\
\mbox{\v H}^1\big(X, \Gc^{(j)}_{T^*_{X,-}}\big)&
}
\]
where $\Gc^{(j)}_{T^*_{X, -}; T^*_{Y,-}}\subset \Gc^{(j)}_{T^*_{Y,-}}$ comprise those automorphisms which restrict to automorphisms 
 in $\Gc^{(j)}_{T^*_{X, -}}$. The following definition is given in \cite{BETTEMB}. It appears in \cite{DW1} as a proposition.

\begin{DEF}\label{f748gf84hf79h8hf}
\emph{Let $(\Uc\ra \Xfr)$ and $(\Vc\ra \Yfr)$ be supermanifolds modelled on $(X, T^*_{X, -})$ and $(Y, T^*_{Y, -})$ respectively and with respective splitting types $j$ and $j^\p$. An embedding $(\Uc\ra \Xfr)\subset (\Vc\ra \Yfr)$ of splitting type $(j^{\p\p}, j, j^\p)$, with $j^{\p\p} \leq j, j^\p$, is defined to be a class $c_{(\Uc\ra \Xfr)/(\Vc\ra\Yfr)}\in \mbox{\v H}^1\big(Y, \Gc^{(j^{\p\p})}_{T^*_{X, -}; T^*_{Y,-}}\big)$ such that:
\begin{enumerate}[(i)]
	\item $r_*\big(c_{(\Uc\ra \Xfr)/(\Vc\ra\Yfr)}\big) = i_*[(\Uc\ra \Xfr)]$ and;
	\item $u_*\big(c_{(\Uc\ra \Xfr)/(\Vc\ra\Yfr)}\big) = i^\p_*\big([\Vc\ra \Yfr]\big)$;
\end{enumerate}
where $i_*: \mbox{\v H}^1\big(X, \Gc^{(j)}_{T^*_{X,-}}\big) \ra  \mbox{\v H}^1\big(X, \Gc^{(j^{\p\p})}_{T^*_{X,-}}\big)$ and $i^\p_* : \mbox{\v H}^1\big(Y, \Gc^{(j^{\p})}_{T^*_{Y,-}}\big) \ra  \mbox{\v H}^1\big(Y, \Gc^{(j^{\p\p})}_{T^*_{Y,-}}\big)$ are the maps on cohomology induced by the respective inclusion of sheaves $\Gc^{(j)}_{T^*_{X,-}}\subset \Gc^{(j^{\p\p})}_{T^*_{X,-}}$ and $\Gc^{(j^{\p})}_{T^*_{Y,-}}\subset \Gc^{(j^{\p\p})}_{T^*_{Y,-}}$.}
\end{DEF}

\noindent
The diagram encoding the rudiments of Definition \ref{f748gf84hf79h8hf} is:
\begin{align}
\xymatrix{
&& \mbox{\v H}^1\big(Y, \Gc^{(j^{\p})}_{T^*_{Y,-}}\big)\ar[d]^{i^\p_*} 
\\
&\mbox{\v H}^1\big(Y, \Gc^{(j^{\p\p})}_{T^*_{X, -}; T^*_{Y,-}}\big) \ar[d]_{r_*} \ar[r]^{u_*} & \mbox{\v H}^1\big(Y, \Gc^{(j^{\p\p})}_{T^*_{Y,-}}\big)
\\
\mbox{\v H}^1\big(X, \Gc^{(j)}_{T^*_{X,-}}\big) \ar[r]^{i_*} &\mbox{\v H}^1\big(X, \Gc^{(j^{\p\p})}_{T^*_{X,-}}\big)
}
\label{rhf74g784hhjp3po3kdp}
\end{align}
The cohomology set $\mbox{\v H}^1\big(Y, \Gc^{(j^{\p\p})}_{T^*_{X, -}; T^*_{Y,-}}\big)$ is a pointed set with base-point represented by the embedding of split models $S(X, T^*_{X, -})\subset S(Y, T^*_{Y, -})$. The maps $u_*$ and $r_*$ are then maps of pointed sets. We intend to apply this notion of embeddings to study families of supermanifolds.

\begin{REM}
\emph{To avoid degenerate cases in the statements of our results we will assume that a supermanifold with splitting type $j^\p$ does \emph{not} have splitting type $j^\p+1$ unless $j^\p = \8$. In particular, the obstruction to splitting such a supermanifold is necessarily non-vanishing unless $j^\p = \8$.}
\end{REM}

\section{Definitions and Constructions}
\label{fihrhf948hf84jf0j40fj903}

\subsection{Families of Supermanifolds}
From Kodaira in \cite[p. 59-60]{KS} a family of compact, complex manifolds parametrised by a space $B$ requires:
\begin{enumerate}[(i)]
	\item a complex manifold $\Mcl$;
	\item a surjective submersion $f : M\ra B$ such that $f^{-1}(b)\subset M$ is a compact, complex submanifold of $\Mcl$ for all $b\in B$;
	\item that the rank of the Jacobian of $f$ is constant.
\end{enumerate}
Condition (iii) above allows for the construction of an atlas of local coordinates describing points in the base and points in the fiber. The generalisation to supermanifolds is immediate.

\begin{DEF}\label{fbvryuvuyviuoijfioejfioe}
\emph{Fix a model $(X, T^*_{X, -})$.  A \emph{family of complex supermanifolds modelled on $(X, T^*_{X,-})$ and parametrised by $B$} is a complex supermanifold $(\Uc\ra \Xc)$ together with:
\begin{enumerate}[(i)]
	\item a morphism $f : \Xc\ra B$ with reduced part defining an analytic family of compact, complex manifolds over $B$;
	\item for each $b\in B$, there exists a complex supermanifold $(\Uc_b\ra\Xc_b)$ modelled on $(X, T^*_{X, -})$ with splitting type $j_b$ together with an embedding $(\Uc_b\ra\Xc_b)\subset (\Uc\ra\Xc)$.
\end{enumerate}
A family of supermanifolds modelled on $(X, T^*_{X, -})$ will be denoted $(\Uc\ra\Xc \ra B)$. The supermanifold $(\Uc\ra\Xc)$ will be referred to as the \emph{total space} of the family; its splitting type will be referred to as \emph{the splitting type of the family}; and the modelling data for $(\Uc\ra\Xc)$ will be referred to as the \emph{total space model}}. 
\end{DEF}

\begin{REM}\label{huirgfihfuhfoieoej}
\emph{The data in a family of supermanifolds $(\Uc\ra\Xc\ra B)$ defines, over each $b\in B$ a triple of integers $(j_b^{\p\p}, j_b, j_b^\p)$, the splitting type of the fiber-wise embedding $\Xc_b\subset \Xc$ (c.f., Definition \ref{f748gf84hf79h8hf}). Of this triple only $j_b^\p$, the splitting type of the total space $(\Uc\ra\Xc)$, will be \emph{a priori} independent of $b\in B$. Hence it makes sense to term $j_b^\p = j^\p$ the `splitting type of the family' as in Definition \ref{fbvryuvuyviuoijfioejfioe}. One of the central objectives of this article is to show that $j_b$ is also `generally' independent of $b$ when the family is a `$gtm$-family' over a Stein base.}
\end{REM}

\noindent
Morally, the difference between a family and a deformation is that the base of the family is pointed. That is, a deformation is a family over a pointed base. This is the distinction we adopt here leading to the notion of a `splitting type deformation'.

\begin{DEF}\label{fvuirviurvieiojowjwo}
\emph{Let $(\Ufr\ra\Xfr)$ be a supermanifold modelled on $(X, T^*_{X, -})$ with splitting type $j$. A \emph{splitting type deformation} of $(\Ufr\ra\Xfr)$ is a family of supermanifolds $(\Uc\ra\Xc\stackrel{f}{\ra} B)$ modelled on $(X, T^*_{X, -})$ together with a distinguished point $b_0\in B$ such that $(\Ufr\ra\Xfr) = f^{-1}(b_0)$. Splitting type deformations will be denoted $(\Uc\ra\Xc\ra (B, b_0))$}.
\end{DEF}

\subsection{Classifying Diagrams}
With the characterisation of embeddings given earlier, we will obtain the following commutative diagram from any family $(\Uc\ra\Xc\ra B)$ with splitting type $j^\p$:
\begin{align}
\xymatrix{
\ar@{-->}[dd]_{\Phi_{(\Uc\ra\Xc)}}B \ar@{-->}[rr]|{[(\Uc\ra \Xc)]} \ar@{-->}[dr]_{\Psi_{(\Uc\ra\Xc)}}&& \mbox{\v H}^1\big(Y, \Gc^{(j^{\p})}_{T^*_{Y,-}}\big)\ar[d]^{i^\p_*} 
\\
&\prod_{j^{\p\p}\leq j^\p}\mbox{\v H}^1\big(Y, \Gc^{(j^{\p\p})}_{T^*_{X, -}; T^*_{Y,-}}\big) \ar[d]_{r_*} \ar[r]^{u_*} & \prod_{j^{\p\p}\leq j^\p}\mbox{\v H}^1\big(Y, \Gc^{(j^{\p\p})}_{T^*_{Y,-}}\big)
\\
\prod_{j\leq j^\p} \mbox{\v H}^1\big(X, \Gc^{(j)}_{T^*_{X,-}}\big) \ar[r]^{i_*} &\prod_{j^{\p\p}\leq j^\p}\mbox{\v H}^1\big(X, \Gc^{(j^{\p\p})}_{T^*_{X,-}}\big)
}
\label{furibvyvuieoijodjo}
\end{align}
where: $(Y, T^*_{Y, -})$ is the total space model; the notation in the horizontal map $B \ra  \mbox{\v H}^1\big(Y, \Gc^{(j^{\p})}_{T^*_{Y,-}}\big)$ means this map is constant and maps all $b\in B$ to the class of the total space $[(\Uc\ra \Xc)]$ (c.f., Remark \ref{huirgfihfuhfoieoej}); the map $\Psi_{(\Uc\ra\Xc)}$ sends $b$ to the class of the fiberwise embedding $(\Uc_b \ra \Xc_b)\subset (\Uc\ra \Xc)$; and $\Phi_{(\Uc\ra\Xc)}$ sends $b\mapsto [(\Uc_b\ra\Xc_b)]$, the class of the fiber over $b$ as a supermanifold modelled on $(X, T^*_{X, -})$. Hence, for each $b$, we get the following diagram:
\begin{align}
\xymatrix{
\ar@{-->}[dd]|{[(\Uc_b \ra \Xc_b)]} \{b\} \ar@{-->}[dr]|{c_{(\Uc_b \ra \Xc_b)/(\Uc\ra\Xc)}} \ar@{-->}[rr]|{[(\Uc\ra \Xc)]} && \mbox{\v H}^1\big(Y, \Gc^{(j^{\p})}_{T^*_{Y,-}}\big)\ar[d]^{i^\p_*} 
\\
&\mbox{\v H}^1\big(Y, \Gc^{(j_b^{\p\p})}_{T^*_{X, -}; T^*_{Y,-}}\big) \ar[d]_{r_*} \ar[r]^{u_*} & \mbox{\v H}^1\big(Y, \Gc^{(j_b^{\p\p})}_{T^*_{Y,-}}\big)
\\
\mbox{\v H}^1\big(X, \Gc^{(j_b)}_{T^*_{X,-}}\big) \ar[r]^{i_*} &\mbox{\v H}^1\big(X, \Gc^{(j_b^{\p\p})}_{T^*_{X,-}}\big)
}
\label{rf4uyfg78f893h8f309fj3}
\end{align}
Note that $j^\p$ in \eqref{furibvyvuieoijodjo} and \eqref{rf4uyfg78f893h8f309fj3}, the splitting type of the total space $(\Uc\ra\Xc)$, is independent of $b$---in contrast to the other splitting types $j_b$ and $j_b^{\p\p}$ in \eqref{rf4uyfg78f893h8f309fj3}. This is in accordance with Remark \ref{huirgfihfuhfoieoej}. The diagrams \eqref{furibvyvuieoijodjo} and \eqref{rf4uyfg78f893h8f309fj3} can be derived from a given family of supermanifolds and so we term them in the following.

\begin{DEF}
\emph{Let $(\Uc\ra\Xc\ra B)$ be a family of supermanifolds modelled on $(X, T^*_{X, -})$. The diagram in \eqref{furibvyvuieoijodjo} will be referred to as the \emph{classifying diagram for the family}. The diagram in \eqref{rf4uyfg78f893h8f309fj3} is referred to as the \emph{classifying diagram for the family over $b$}.}
\end{DEF}

\noindent
On inspection of the classifying diagram of a family $(\Uc\ra\Xc\ra B)$ over $b\in B$, we will obtain the following general relation between the splitting types.

\begin{LEM}\label{fbiuihfiuhf83fh3f3j903}
Let $(\Uc\ra\Xc\ra B)$ be a family of supermanifolds with splitting type $j^\p$. Over each $b\in B$ let $(j_b^{\p\p}, j_b, j^{\p})$ be the splitting type of the fiber-wise embedding, where $j_b$ is the splitting type of the fiber as a supermanifold. Then,
\begin{align*}
j_b^{\p\p} \leq \min\{j_b, j^\p\}.
\end{align*}\qed
\end{LEM}

\begin{REM}
\emph{Note in particular that there need not exist any other relations between the splitting type of the fiber $j_b$ and the splitting type of the family $j^\p$.}
\end{REM}

\subsection{Classifying Maps}
The maps $\Psi_{(\Uc\ra\Xc)}$ and $\Phi_{(\Uc\ra\Xc)}$ appearing in classifying diagrams will be referred to as `classifying maps'. To distinguish these maps, we will refer to $\Phi_{(\Uc\ra\Xc)}$ as the \emph{fiber-wise classifying map}. These are maps of sets but they need \emph{not} be maps of pointed sets since we are not assuming $B$ is a pointed set and nor that the image of these maps contain the base-point. The following result guarantees when $\Psi_{(\Uc\ra\Xc)}$ will \emph{not} be a map of pointed sets. Before we present it, we will need to recall the notion of `weak non-splitness' from \cite{BETTHIGHOBS}.

\begin{DEF}\label{rf4fg9hf9hf98h30f03j0f3}
\emph{A supermanifold $(\Ufr\ra\Xfr)$ modelled on $(X, T^*_{X, -})$ of splitting type $j$ is said to be \emph{weakly non-split} if:
\begin{enumerate}[(i)]
	\item its class $[(\Ufr\ra\Xfr)]\in \mbox{\v H}^1\big(X, \Gc^{(j)}_{T^*_{X, -}}\big)$ does not coincide with the base-point;
	\item the image of the class $[(\Ufr\ra\Xfr)]$ under the set-theoretic mapping of pointed sets $\mbox{\v H}^1\big(X, \Gc^{(j)}_{T^*_{X, -}}\big)\ra \mbox{\v H}^1\big(X, \Gc^{(j^\p)}_{T^*_{X, -}}\big)$ does not coincide with the base-point for any $j>j^\p$.
\end{enumerate}
A supermanifold that is \emph{not} weakly non-split is said to be \emph{strongly split}.}
\end{DEF}

\begin{REM}\label{tbvuriuvoevneioicon}
\emph{A strongly split supermanifold will be split. However, split supermanifolds need not be strongly split. Similarly, a non-split supermanifold will be weakly non-split but not necessarily conversely. See \cite{BETTHIGHOBS} for further discussion on this point.}
\end{REM}
\noindent
We now have the following.

\begin{PROP}\label{rf4h78fh498fh9f038j}
Let $(\Uc\ra \Xc\ra (B, b_0))$ be a splitting type deformation. Suppose $(\Uc\ra \Xc)$ is of splitting type $j^\p$ and weakly non-split. Then $\Psi_{(\Uc\ra\Xc)}$ cannot be a map of pointed sets.
\end{PROP}

\begin{proof}
A map of pointed sets $(A, a) \ra (B, b)$ is a set theoretic mapping $A\ra B$ which preserves base points, i.e., sends $a\mapsto b$. Now, the base of a spitting type deformation $(\Uc\ra\Xc\ra (B, b_0))$ is the pointed set $(B, b_0)$. To see why $\Psi_{(\Uc\ra\Xc)}$ \emph{cannot} preserve base points,
recall from \eqref{rf4uyfg78f893h8f309fj3} the classifying diagram of $(\Uc\ra\Xc\ra B)$ over $b_0$:
\begin{align*}
\xymatrix{
\ar@{-->}[dd]|{(\Uc_{b_0} \ra \Xc_{b_0})} \{b_0\} \ar@{-->}[dr]|{c_{(\Uc_{b_0} \ra \Xc_{b_0})/(\Uc\ra\Xc)}} \ar@{-->}[rr]|{[(\Uc\ra \Xc)]} && \mbox{\v H}^1\big(Y, \Gc^{(j^{\p})}_{T^*_{Y,-}}\big)\ar[d]^{i^\p_*} 
\\
&\mbox{\v H}^1\big(Y, \Gc^{(j_{b_0}^{\p\p})}_{T^*_{X, -}; T^*_{Y,-}}\big) \ar[d]_{r_*} \ar[r]^{u_*} & \mbox{\v H}^1\big(Y, \Gc^{(j_{b_0}^{\p\p})}_{T^*_{Y,-}}\big)
\\
\mbox{\v H}^1\big(X, \Gc^{(j_{b_0})}_{T^*_{X,-}}\big) \ar[r]^{i_*} &\mbox{\v H}^1\big(X, \Gc^{(j_{b_0}^{\p\p})}_{T^*_{X,-}}\big)
}
\end{align*}
where $(Y, T^*_{Y, -})$ is the total space model; and $j_{b_0}^{\p\p} \leq j_{b_0}, j^\p$. If $\Psi_{(\Uc\ra\Xc)}$ were a map of pointed sets then, since $b_0$ is the base-point of $B$, we must have $\Psi_{(\Uc\ra\Xc)}(b_0) = c_{(\Uc_{b_0} \ra \Xc_{b_0})/(\Uc\ra\Xc)}$ be the base-point in $\mbox{\v H}^1\big(Y, \Gc^{(j_{b_0}^{\p\p})}_{T^*_{X, -}; T^*_{Y,-}}\big)$. Since $u_*$ is a map of pointed sets, the composition $u_*\Psi_{(\Uc\ra\Xc)}$ will therefore send $b_0$ to the base point in $\mbox{\v H}^1\big(Y, \Gc^{(j_{b_0}^{\p\p})}_{T^*_{Y,-}}\big)$. The diagram in \eqref{rf4uyfg78f893h8f309fj3} commutes and therefore, recalling that $B \ra \mbox{\v H}^1\big(Y, \Gc^{(j^{\p})}_{T^*_{Y,-}}\big)$ is constant and sends all $b\in B$ to $\{[(\Uc\ra \Xc)]\}$, we see that $i^\p_*\big([(\Uc\ra \Xc)]\big)$ must be the base point in $\mbox{\v H}^1\big(Y, \Gc^{(j_{b_0}^{\p\p})}_{T^*_{Y,-}}\big)$. This contradicts $(\Uc\ra \Xc)$ being weakly non-split (c.f., Definition \ref{rf4fg9hf9hf98h30f03j0f3}(ii)). 
\end{proof}

\begin{COR}\label{rf847gf784hf9hf983}
Let $(\Uc\ra\Xc\ra (B, b_0))$ be a splitting type deformation and suppose the fiber over $b_0$ represents the embedding of split models. Then the total space $(\Uc\ra \Xc)$ will be split.
\end{COR}

\begin{proof}
Since the fiber over $b_0$ represents the embedding of split models, the classifying map $\Psi_{(\Uc\ra\Xc)}$ sents $b_0$ to the base-point. Arguments from the proof of Proposition \ref{rf4h78fh498fh9f038j} then imply $(\Uc\ra \Xc)$ cannot be weakly non-split. Hence it will be strongly split and therefore split as a supermanifold (see Remark \ref{tbvuriuvoevneioicon}).
\end{proof}

\noindent
For sake of clarity we term the class of deformations in Corollary \ref{rf847gf784hf9hf983}.

\begin{DEF}\label{rhf94hf98h80fh309fj0}
\emph{Let $(\Uc\ra \Xc\ra (B, b_0))$ be a splitting type deformation of supermanifolds. This deformation is said to be \emph{centrally split} if its classifying map $\Psi_{(\Uc\ra\Xc)}$ defines a map of pointed sets.}
\end{DEF}

\noindent
By Corollary \ref{rf847gf784hf9hf983} the total space of a centrally split, splitting type deformation will itself be split as a supermanifold. Hence if the total space $(\Uc\ra\Xc)$ is itself \emph{non-split}, the splitting type deformation $(\Uc\ra \Xc\ra (B, b_0))$ \emph{cannot} be centrally split. Now, it is of course still possible for the fiber over $b_0$, $(\Uc_{b_0}\ra \Xc_{b_0})$ to be split as a supermanifold. Hence that the fiber-wise classifying map $\Phi_{(\Uc\ra\Xc)}$ in \eqref{furibvyvuieoijodjo} will be a map of pointed sets. Generalising Definition \ref{rhf94hf98h80fh309fj0} accordingly, we have:

\begin{DEF}\label{tgv784gf7h498fh30jf39j}
\emph{Let $(\Uc\ra \Xc\ra (B, b_0))$ be a splitting type deformation of supermanifolds. It is said to be \emph{weakly centrally split} if its fiber-wise classifying map $\Phi_{(\Uc\ra\Xc)}$ defines a map of pointed sets.}
\end{DEF}

\begin{EX}\label{fnvrbvuirnvionieovoemv}
In the case where $B$ is the superspace $\Cbb^{0|1}$, Donagi and Witten in \cite{DW1} construct a deformation of a super Riemann surface $\Scl$ which cannot be be centrally split since, by their construction, its total space is non-split. This construction was generalised to $B = \Cbb^{0|n}$ and studied further by the author in \cite{BETTSRS}. Again, the deformations of super Riemann surfaces $(\Uc\ra\Xc\ra (\Cbb^{0|n}, 0))$ appearing in \cite{BETTSRS} are generally not centrally split.
\end{EX}

\begin{REM}
\emph{A subtle point to note concerning the deformations of super Riemann surfaces $\Scl$ discussed in Example \ref{fnvrbvuirnvionieovoemv} is: while these deformations are not centrally split, the fiber $\Scl$ over the base-point $0\in \Cbb^{0|n}$ will always be split as a supermanifold. This is because $\Scl$ is a $(1|1)$-dimensional supermanifold and all such supermanifolds are split. Hence, according to Definition \ref{tgv784gf7h498fh30jf39j}, these deformations are \emph{weakly centrally split}. That the deformation is \emph{not necessarily} centrally split is more a reflection of the fiberwise embedding of $\Scl$.}
\end{REM}

\begin{REM}\label{fkvnkbvbvnvinvoinenvoe}
\emph{It should be noted that the families and deformations considered so far have been over a reduced (i.e., \emph{non} super) space, in contrast to the deformations of super Riemann surfaces in Example \ref{fnvrbvuirnvionieovoemv}. All the definitions given so far can be generalised to families and deformations over superspaces. We will be vindicated on this point in the second part of this article.}
\end{REM}

\section{\emph{ltm}-, \emph{ism}- and \emph{gtm}-Families}
\label{fvirubvuirbiurburnornoir}

\subsection{Local Triviality: Generalities}
Topological properties of supermanifolds are defined by reference to their reduced space. Hence a morphism $\Xc\ra B$ will be a surjective submersion if the corresponding map on reduced spaces $f: Y\ra B$ is a surjective submersion (c.f., Definition \ref{fbvryuvuyviuoijfioejfioe}(i)). Ehresmann's Fibration Lemma then asserts that $(Y\stackrel{f}{\ra} B)$ is locally trivial, topologically. Hence, locally, the fiber $X_b \stackrel{\Delta}{=} f^{-1}(b)\subset Y$ is homeomorphic to $X_{b^\p}$ for any $b, b^\p\in W$, where $W\subset B$ a neighbourhood of $b$, sufficiently small. Hence $Y|_{f^{-1}(W)} \cong X\times W$, for $X$ the fiber over some $b\in W$. Now if $Y\stackrel{f}{\ra} B$ is a family of compact, complex manifolds, the fiber $X_b = f^{-1}(b)$ will be endowed with a complex structure. This fiberwise complex structure can vary non-trivially as we (infinitesimally) vary $b$ and so it certainly need \emph{not} be the case that the local homeomorphism $Y|_{f^{-1}(W)} \stackrel{\cong}{\ra} X\times W$ will lift to a biholomorphism. It need not even induce a biholomorphism of fibers $X_b, b\in W$. A classical result of Fisher and Grauert, taken from \cite[p. 269]{HUYB}, concerns the case when it will:

\begin{THM}\label{tg784f87h9fh38fh03}
Let $(Y\stackrel{f}{\ra} B)$ be a family of compact, complex manifolds.\footnote{Recall, this means both $Y$ and $B$ are complex manifolds and $f: Y\ra B$ is a holomorphic, surjective submersion.} For each $b, b^\p$ in an open set $W\subset B$, if $X_b \cong X_{b^\p}$ as complex manifolds then, upon shrinking $W$ if necessary, there exists a biholomorphism $Y|_{f^{-1}(W)} \stackrel{\cong}{\ra} X\times W$, where $X$ is the fiber over some $b\in W$. 
\qed
\end{THM}

\noindent
If the family $(Y\ra B)$ satisfies the hypotheses in Theorem \ref{tg784f87h9fh38fh03} for any $b\in B$ it is referred to as `locally trivial' or `isotrivial'. More precisely:

\begin{DEF}\label{fjnvrbvbvknrvnlkeee}
\emph{Let $(Y\ra B)$ be a family of compact, complex manifolds. For any $b, b^\p\in B$, if the fibers $X_b$ and $X_{b^\p}$ are isomorphic, the family is referred to as \emph{isotrivial}.}
\end{DEF}

\noindent
For $(Y\ra B)$ an isotrivial family of compact, complex manifolds and any $b\in B$, we have an isomorphism of complex manifolds $X_b \cong X$ for some complex manifold $X$. We refer to $X$ as the \emph{typical fiber} of $(Y\ra B)$. Theorem \ref{tg784f87h9fh38fh03} asserts that any isotrivial family is, locally, a product of complex manifolds.

\begin{DEF}
\emph{Let $(Y\ra B)$ be an isotrivial family with typical fiber $X$. The family is said to be \emph{globally trivial} if it is biholomorphic to the product $X\times B$.}
\end{DEF}

\subsection{Families of Sheaves}
The following definition concerning sheaves on isotrivial families of compact, complex manifolds will be relevant for applications to supermanifolds.

\begin{DEF}
\emph{Let $(Y\ra B)$ be an isotrivial family of compact, complex manifolds with typical fiber $X$. Fix a sheaf $\Gc$ on $X$. A sheaf $\Fc$ on $Y$ is said to define \emph{a family of sheaves over $\Gc$} if, for each $b\in B$, we have a surjection of sheaves on $X$, $\vp_b : i^*_b\Fc \ra \Gc \ra 0$ where $i_b : X\subset Y$ is the embedding of the fiber over $b$. We denote a family $\Fc$ of sheaves over $\Gc$ by $\Fc_{/\Gc}$.}
\end{DEF}

\noindent
The prototypical example motivating the above definition of families of sheaves is the following.

\begin{EX}\label{rfh4f79hf983h0fj039}
Suppose $(Y\ra B)$ is globally trivial with typical fiber $X$. This means $Y \cong X\times B$. Denote by $p_X: Y\ra X$ the projection onto the first factor. Then for each $b$ we have a commutative diagram of spaces,
\[
\xymatrix{
X \ar@{=}[dr]_{{\bf 1}_X} \ar@{^{(}->}[r]^{i_b} & Y\ar[d]^{p_X}
\\
& X
}
\]
where $i_b$ is the inclusion $x \mapsto (x, b)$. Fix a sheaf $\Gc$ on $X$ and let $\Fc = p_X^*\Gc$. Since $p_Xi_b = {\bf 1}_X$ we have,
\[
i_b^*\Fc = i_b^*p_X^*\Gc = (p_Xi_b)^*\Gc \cong {\bf 1}_X^*\Gc = \Gc.
\]
Hence when $Y$ is globally trivial, any sheaf on $Y$ which is pulled back from a sheaf $\Gc$ on $X$ will define a family of sheaves over $\Gc$, i.e., for any sheaf $\Gc$ on $X$ we have a family $p_X^*\Gc_{/\Gc}$.
\end{EX}

\subsection{Families of Models and Supermanifolds}

\subsubsection{Families of Models}
From Definition \ref{fbvryuvuyviuoijfioejfioe}(i), if $(\Xc\ra B)$ is a family of supermanifolds with $\Xc$ modelled on $(Y, T^*_{Y, -})$ then we will have a family of compact, complex manifolds $Y\ra B$. Indeed, there will exist a commutative diagram of spaces:
\begin{align}
\xymatrix{
\ar@{^{(}->}[d] Y \ar[r] & B\ar@{=}[d]
\\
\Xc\ar[r] & B
}
\label{rhf793hf93hfh83h0}
\end{align}
This leads to the following general characterisation:

\begin{PROP}\label{rh4fh9f83jf09j393}
Let $(\Uc\ra \Xc\ra B)$ be a family of supermanifolds modelled on $(X, T^*_{X, -})$ with total space model $(Y, T^*_{Y, -})$. Then $(Y\ra B)$ is an isotrivial family of compact, complex manifolds; and $T^*_{Y, -}$
is a family of locally free sheaves over $T^*_{X, -}$.
\end{PROP}

\begin{proof}
To each $b$ we have from \eqref{rhf793hf93hfh83h0} a commutative diagram,
\[
\xymatrix{
X_b \ar@{^{(}->}[r] \ar@{^{(}->}[d] & \ar@{^{(}->}[d] Y \ar[r] &B \ar@{=}[d]
\\
\Xc_b \ar@{^{(}->}[r] \ar[r] & \Xc\ar[r] & B
}
\] 
where $X_b$ and $Y$ are the reduced spaces of the supermanifolds $\Xc_b$ and $\Xc$ respectively. Now by assumption $(\Uc_b \ra\Xc_b)$ is modelled on $(X, T^*_{X, -})$. Hence for each $b$ we must have $X_b \cong X$. Hence, any two fibers of $(Y\ra B)$ will be isomorphic as complex manifolds and so, by Definition \ref{fjnvrbvbvknrvnlkeee}, $(Y\ra B)$ will be an isotrivial family with typical fiber $X$. To see that $T_{Y, -}^*$ will define a family of sheaves over $T_{X, -}^*$ recall that over each $b$, we have a supermanifold modelled on $(X, T^*_{X, -})$ embedding into a supermanifold modelled on $(Y, T^*_{Y, -})$. Hence we have an embedding of models $(X_b, T_{X_b, -}^*) \cong (X, T^*_{X, -})\subset (Y, T^*_{Y, -})$ which, by definition, comprises an embedding $i_b: X_b\cong X \subset Y$ and a surjection $i_b^*T^*_{Y, -} \ra T^*_{X_b, -}\cong T_{X, -}^*$ (see \eqref{ffgggruepepedjfhffhfj}). 
\end{proof}

\subsubsection{Families of Split Models}
As an application of Example \ref{rfh4f79hf983h0fj039} and Proposition \ref{rh4fh9f83jf09j393}, we will present here a construction of a family of split supermanifolds. Let $Y = X\times B$ be globally trivial and denote by $p_X$ resp., $p_B$ the projections onto $X$ resp., $B$. Fix a locally free sheaf $T_{X, -}^*$ on $X$, Then by Example \ref{rfh4f79hf983h0fj039} we know that $T_{Y, -}^* \stackrel{\Delta}{=} p_X^*T^*_{X, -}$ will be a family of sheaves over $T_{X, -}^*$. It is locally free since $T_{X, -}^*$ is locally free. 

\begin{LEM}\label{f78rgf874hf93h9f83}
\emph{The projection $p_B : Y\ra B$ defines a morphism of supermanifolds $S(p_B): S(Y, T^*_{Y, -})\ra B$.\footnote{Any manifold $B$ with structure sheaf $\Oc_B$ can be thought of as a supermanifold upon noting that $\Oc_B \cong \wedge^\bt_{\Oc_B}{\bf 0}_B$ as algebras, where ${\bf 0}_B$ is a sheaf on $B$ of rank zero.}}
\end{LEM}

\begin{proof}
The split model $S(Y, T^*_{Y, -})$ is split so, in particular, it is projected. Hence there exists a projection map $\pi : S(Y, T^*_{Y, -})\ra Y$. The desired morphism $S(p_B): S(Y, T^*_{Y, -})\ra B$ is then the composition $S(Y, T^*_{Y, -})\stackrel{\pi}{\ra} Y\stackrel{p_B}{\ra} B$.
\end{proof}

\noindent 
Appealing to the locally ringed space definition of supermanifolds, we see that the fiber of the morphism $S(p_B)$ over $b$ is, 
\begin{align}
\big(p_B^{-1}(b), \Oc_{S(Y, T^*_{Y, -})}|_{p_B^{-1}(b)}\big)
&\cong 
(X_b, i_b^*\wedge^\bt T^*_{Y, -}) 
\notag
\\
&= 
(X_b, \wedge^\bt i_b^*T^*_{Y, -}) 
\cong
S(X, T^*_{X, -}).
\label{rhf894hf94hf9803}
\end{align}
Thus we have a family of split supermanifolds modelled on $(X, T^*_{X, -})$.

\subsection{Isotrivial Models and Supermanifolds}
Let $(Y\stackrel{f}{\ra}B)$ be an isotrivial family with typical fiber $X$. Then by Theorem \ref{tg784f87h9fh38fh03} we know that for any sufficiently small open set $W\subset B$ we have an isomorphism $\xi_W: Y|_{f^{-1}(W)}\stackrel{\sim}{\ra} X\times W$. Hence, over $W$ we have projections $Y|_{f^{-1}(W)}\stackrel{p_W\xi_W}{\ra} W$ and $Y|_{f^{-1}(W)} \stackrel{p_X\xi_W}{\ra}X$. We will refer to such an open set as a \emph{trivialising open set} for the isotrivial family $(Y\stackrel{f}{\ra}B)$. 

\begin{DEF}\label{rhf784f794hf893h}
\emph{Let $(Y\stackrel{f}{\ra}B)$ be an isotrivial family with typical fiber $X$. Fix a sheaf $\Gc$ on $X$ and let $\Fc$ be a family of sheaves over $\Gc$. We say $\Fc_{/\Gc}$ is:
\begin{enumerate}[(i)]
	\item an \emph{isotrivial family} if, for any $b\in B$, the morphism $i_b^*\Fc\ra\Gc$ is an isomorphism;
	\item a \emph{locally trivial family} if, over any trivialising open set $W\subset B$, there exists an isomorphism $\Fc|_{f^{-1}(W)} \cong \xi_W^*p_X^*\Gc$ where $\xi_W$ is the biholomorphism $Y|_{f^{-1}(W)}\stackrel{\cong}{\ra} X\times W$.
\end{enumerate}}
\end{DEF}

\begin{LEM}\label{ruihfhfh398fh3fh93}
A locally trivial family of sheaves will be isotrivial.
\end{LEM}

\begin{proof}
Let $(Y\stackrel{f}{\ra} B)$ be an isotrivial family with typical fiber $X$. Fix a sheaf $\Gc$ on $X$ and let $\Fc_{/\Gc}$ be a locally trivial family of sheaves. Fix a point $b\in B$ and denote by $i_b : X\subset Y$ the inclusion of the fiber over $b$. We want to show that $i_b^*\Fc\cong \Gc$. To show this firstly recall, by isotriviality of $(Y\ra B)$, that there will exist a trivialising neighbourhood $W_b\subset B$ of $b$. Now $\Fc$ is a sheaf on $Y$. Since sheaves are compatible with restrictions to subsets we have $i_b^*\Fc = i_b^*\Fc|_{f^{-1}(W)}$. By local triviality of $\Fc_{/\Gc}$ we have the isomorphism $\Fc|_{f^{-1}(W)}\cong \xi_W^*p_X^*\Gc$. Now note that the following diagram will commute:
\[
\xymatrix{
Y|_{f^{-1}(W)} \ar[r]^{\xi_W} & X\times W\ar[d]^{p_X}
\\
\ar[u]^{i_b}X  \ar@{=}[r]_{{\bf 1}_X} & X
}
\]
This gives,
\begin{align*}
i_b^*\Fc 
=
i_b^*\Fc|_{f^{-1}(W)}
\cong
i_b^* \xi_W^*p_X^*\Gc
\cong 
{\bf 1}_X^*\Gc
= \Gc.
\end{align*}
Hence $\Fc_{/\Gc}$ will be an isotrivial family of sheaves in the sense of Definition \ref{rhf784f794hf893h}(i).
\end{proof}

\noindent
\begin{REM}
\emph{For a family of complex manifolds, isotriviality will imply local triviality. This is the implication of Fisher and Grauert's result in Theorem \ref{tg784f87h9fh38fh03}. It is unclear as to whether such a statement will hold for families of sheaves. Lemma \ref{ruihfhfh398fh3fh93} illustrates that, at the very least, local triviality is a stronger condition than isotriviality.}
\end{REM}

\noindent
The existence of a locally trivial family of sheaves is a separate question and may be of independent interest. Assuming this question can be resolved in the affirmative, it is straightforward to then adapt the construction in the previous section, of families of split supermanifolds over trivial families $(Y\ra B)$, to obtain families over isotrivial families. More precisely:

\begin{PROP}\label{rfb478fg7hf983f03j0}
Let $(Y\stackrel{f}{\ra}B)$ be an isotrivial family with typical fiber $X$. Fix a locally free sheaf $T_{X, -}^*$ on $X$. Suppose $(T_{Y, -}^*)_{/T^*_{X, -}}$ is a locally trivial family. Then $(Y\stackrel{f}{\ra}B)$ defines a family of split models $S(Y, T^*_{Y, -})\ra B$ with typical fiber $S(X, T^*_{X, -})$.
\end{PROP}

\begin{proof}
Arguing as in Lemma \ref{f78rgf874hf93h9f83}, the fibration $f:Y\ra B$ defines a fibration of supermanifolds $S(f):S(Y, T^*_{Y,-})\ra B$. To see that the typical fiber will be $S(X, T^*_{X, -})$ let $W\subset B$ be a trivialising open set. Over $W$ we have by restriction the family $S\big(Y|_{f^{-1}(W)}, T^*_{Y, -}|_{f^{-1}(W)}\big)\ra W$. Now by assumption $T^*_{Y, -}|_{f^{-1}(W)} \cong \xi_W^*p_X^*T_{X, -}$ where $\xi_W: Y|_{f^{-1}(W)}\stackrel{\cong}{\ra}X\times W$ is the trivialisation. Therefore, by the construction in the previous section (c.f., \eqref{rhf894hf94hf9803}),
the fiber over a point in $W$ will be $S(X, T^*_{X, -})$. Since this holds in a sufficiently small neighbourhood of any point in $B$, the proposition follows.
\end{proof}

\subsection{Definitions}
We conclude now with definitions serving to classify families of models and supermanifolds.

\begin{DEF}\label{rgf784gf398fh30}
\emph{Let $Y\ra B$ be an isotrivial family with typical fiber $X$. Fix a locally free sheaf $T^*_{X, -}$ on $X$ and a family of sheaves $(T_{Y, -}^*)_{/T^*_{X, -}}$ where $T^*_{Y, -}$ is a locally free sheaf on $Y$. The model $(Y, T^*_{Y, -})$ is said to be:
\begin{enumerate}[(i)]
	\item \emph{locally trivial} if $(T_{Y, -}^*)_{/T^*_{X, -}}$ is locally trivial;
	\item \emph{isotrivial} if $(T_{Y, -}^*)_{/T^*_{X, -}}$ is isotrivial;
	\item  \emph{globally trivial} if $Y = X\times B$ is globally trivial and $T^*_{Y, -}= p_X^*T^*_{X, -}$.
\end{enumerate}}
\end{DEF}

\begin{DEF}\label{rf784gf794hf984hf84hf0}
\emph{Let $(\Uc\ra \Xc\ra B)$ be a family of supermanifolds modelled on $(X, T^*_{X,-})$ with total space model $(Y, T^*_{Y,-})$.  This family will be referred to as an \emph{ltm}- resp., \emph{ism}- resp., \emph{gtm-family} if its total space model $(Y, T^*_{Y, -})$ is locally trivial resp., isotrivial resp., globally trivial.}
\end{DEF}

\noindent
By definition of local triviality of a family of sheaves in Definition \ref{rhf784f794hf893h}(ii) we see: if $(\Uc\ra \Xc\ra B)$ is an \emph{ltm}-family of supermanifolds, then over any trivialising open set $W\subset B$ the restriction $(\Uc\ra \Xc\ra B)|_W = \big(\Uc_{f^{-1}(W)}, \Xc|_{f^{-1}(W)}, W\big)$ will be a \emph{gtm}-family. Hence the notion of \emph{ltm} can be defined by reference to \emph{gtm} and restriction, i.e., that an \emph{ltm}-family is a family that is locally a \emph{gtm}-family. Clearly, any \emph{gtm}-family will be an \emph{ltm}-family. Now by Lemma \ref{ruihfhfh398fh3fh93} we see that any \emph{ltm}-family will be an \emph{ism}-family. These observations are collected in the following.

\begin{COR}\label{thg784hg94h98j03}
Any gtm-family of supermanifolds will be an ltm-family; and any ltm-family of supermanifolds will be an ism-family.\qed
\end{COR}

\noindent
In the next section we will introduce a notion of analyticity for families. Our objective will be to show that the splitting type of \emph{analytic} $gtm$-families over a Stein base $B$ will \emph{not} vary locally.

\section{The Splitting Type of $gtm$-Families}
\label{tgh748g749hg84093j93j333}

\subsection{Analyticity}
The definition of a family of supermanifolds given in Definition \ref{fbvryuvuyviuoijfioejfioe} was inspired by the analogous definition for complex manifolds as it appears in \cite{KS}. Subsequently, we derived the `classifying diagram' for a family of complex supermanifolds modelled on $(X, T^*_{X, -})$ in \eqref{furibvyvuieoijodjo}. The notion of analyticity for these families is unclear since the classifying maps $\Psi_{(\Uc\ra\Xc)}$ and $\Phi_{(\Uc\ra\Xc)}$ are understood as maps of sets or pointed sets. Note however that we have projections onto the obstruction space, which is a complex vector space. To see this, fix an embedding of models $(X, T^*_{X, -})\subset (Y, T^*_{Y, -})$. Let $\Ac_{T^*_{X, -}} = \oplus_j\Ac_{T^*_{X, -}}^{(j)}$ and $\Ac_{T^*_{Y, -}} = \oplus_j\Ac_{T^*_{Y, -}}^{(j)}$ denote the respective obstruction sheaves of the models; and $\Ac_{T^*_{X, -};T^*_{Y, -}} = \oplus_j \Ac_{T^*_{X, -}; T^*_{Y, -}}^{(j)}$ the sheaf associated to the embedding\footnote{this sheaf is $\Ac_{T^*_{X, -};T^*_{Y, -}}^{(j)} = \Gc_{T^*_{X, -}; T^*_{Y, -}}^{(j)}/\Gc_{T^*_{X, -}; T^*_{Y, -}}^{(j+1)}$.}. Then as shown in \cite{DW1} and \cite{BETTEMB}, there exists a commutative diagram extending the classifying diagram one step further,
\begin{align}
\xymatrix{
\ar[d] \prod_j \mbox{\v H}^1\big(Y, \Gc^{(j)}_{T^*_{X, -}; T^*_{Y -}}\big) \ar[dr] \ar[r]  &\prod_j \mbox{\v H}^1\big(Y, \Gc^{(j)}_{T^*_{Y, -}}\big)\ar[dr]
\\
\prod_j \mbox{\v H}^1\big(X, \Gc^{(j)}_{T^*_{X, -}}\big)\ar[dr] & H^1\big(Y, \Ac_{T^*_{X, -};T^*_{Y, -}}\big) \ar[d] \ar[r] & H^1\big(Y, \Ac_{T^*_{Y, -}}\big)\ar[d]
\\
& H^1\big(X, \Ac_{T^*_{X, -}}\big) \ar[r] & H^1\big(X, \Ac_{T^*_{Y, -}}|_X\big)
}
\label{fvnkbvhbrybvunienvimepo}
\end{align}
Donagi and Witten in \cite{DW1} make use of a diagram of the above kind in formulating their `compatibility lemma', establishing a relation between the obstruction class to splitting a submanifold of a supermanifold to that of the ambient supermanifold itself. This prompts the following definition which will be convenient for later purposes. 

\begin{DEF}\label{rgf79f9h9hohiooihfeofoi}
\emph{The diagram of abelian cohomology groups in \eqref{fvnkbvhbrybvunienvimepo}, i.e., the square on the lower right side, will be referred to as the \emph{compatibility diagram associated to the embedding of models $(X, T^*_{X, -})\subset (Y, T^*_{Y, -})$}.}
\end{DEF}

\noindent
Presently, we are interested in the map $\om: \prod_j \mbox{\v H}^1\big(X, \Gc^{(j)}_{T^*_{X, -}}\big)\ra H^1\big(X, \Ac_{T^*_{X, -}}\big)$.

\begin{DEF}\label{rufh9fh38f039jf3}
\emph{Let $(\Uc\ra\Xc\ra B)$ be a family of supermanifolds modelled on $(X, T^*_{X, -})$. The composition 
\[
B \stackrel{\Phi_{(\Uc\ra\Xc)}}{\lra}  \prod_j \mbox{\v H}^1\big(X, \Gc^{(j)}_{T^*_{X, -}}\big)\stackrel{\om}{\lra} H^1\big(X, \Ac_{(X, T^*_{X, -})}\big)
\] 
is referred to as the \emph{splitting type differential} of the family.}
\end{DEF}

\begin{DEF}\label{rfh489f98hfj3039k}
\emph{Let $B$ be a complex analytic space. A family of supermanifolds $(\Uc\ra\Xc\ra B)$ is said to be \emph{analytic} if its splitting type differential is analytic as a map between complex analytic spaces.}
\end{DEF}

\subsection{Locality of Splitting Type}
A function $f : X\ra \Zbb$ is said to be \emph{locally constant at $x$} if there exists an open neighbourhood $V$ of $x$ such that $f|_V$ is constant. The main objective of this section is to prove the followng.

\begin{THM}\label{rg784gf794hf98hf0j3fj34444}
Let $(\Uc\ra\Xc\ra B)$ be a $gtm$-family of supermanifolds over a connected, Stein base $B$ with splitting type $j^\p$. Let $j_b$ denote the splitting type of the fiber over a point $b\in B$. If:
\begin{enumerate}[$\bt$]
	\item the family is analytic and; 
	\item $j_b< \8$;
\end{enumerate}
then the assignment 
\[
b \longmapsto j_b
\]
defines a locally constant function $B\ra \Zbb$ at $b$.
\end{THM}

\subsubsection{A Remark on Theorem \ref{rg784gf794hf98hf0j3fj34444}}
Kodaira in \cite[pp. 195--202]{KS} deduces local triviality of family of complex manifolds $(M_t)_{t\in B}$ under the assumption that the mapping $t\mapsto h^1(M_t, T_{M_t})\stackrel{\Delta}{=}\dim_\Cbb H^1(M_t, T_{M_t})$, for $T_{M_t}$ the tangent sheaf, is constant and the Kodaira-Spencer class of the family vanishes. Conversely, if the family is locally trivial, then by Fischer and Grauert's theorem (Theorem \ref{tg784f87h9fh38fh03}), locally on the base, fibers of the family $(M_t)_{t\in B}$ will be biholomorphic. As a consequence, the mapping $t\mapsto h^1(M_t, T_{M_t})$ will  at least be \emph{locally} constant for such families. It is this last statement which we might observe as being appropriately generalised in Theorem \ref{rg784gf794hf98hf0j3fj34444} to the case of splitting types of supermanifold families. Based on this observation Theorem \ref{rg784gf794hf98hf0j3fj34444} indicates that, more generally\footnote{i.e., for families which are not necessarily of $gtm$-type}, the fiber-wise splitting type of a family of supermanifolds might vary (upper) semicontinuously with respect to parameters on the base, in analogy with the mapping $t\mapsto h^1(M_t, T_{M_t})$ for differentiable families of complex manifolds (c.f., \cite[p. 202]{KS}). We do not attempt to pursue this line of thought any further in this article however. We only mention it as a point of interest and potential future development.

\subsection{Proof of Theorem $\ref{rg784gf794hf98hf0j3fj34444}$}

\subsubsection{Stein Spaces}
In the proof of Theorem \ref{rg784gf794hf98hf0j3fj34444} we will make use of the following properties of Stein spaces. These are standard results and can be found, for instance, in \cite{GRAUST}.

\begin{LEM}\label{ffefrgtg5g5g5g55g5}
Let $B$ be a Stein space. Then
\begin{enumerate}[(i)]
	\item there exists a smooth embedding $B\subset \Cbb^m$ for some $m$;
	\item if $Z\subset B$ is a closed, analytic subspace, then it is Stein;
	\item for any coherent sheaf of $\Oc_B$-modules $\Fc$ on $B$, $H^k(B, \Fc)$ is a complex vector space for each $k$ with,
	\begin{align*}
	\dim_\Cbb H^0(B, \Fc) < \8 &&\mbox{and}&&H^k(B, \Fc) = (0)
	\end{align*}
	for all $k> 0$.
\end{enumerate}\qed
\end{LEM}

\begin{REM}
\emph{The proof of Theorem \ref{rg784gf794hf98hf0j3fj34444} will only require a connected base that satisfies the properties in Lemma \ref{ffefrgtg5g5g5g55g5}. As such, it ought to straightforwardly generalise to supermanifold families defined over connected, affine varieties where analogous properties to Lemma \ref{ffefrgtg5g5g5g55g5} are also known to hold.}
\end{REM}

\subsubsection{The Obstruction Space of a $gtm$-Family}
Let $(\Uc\ra \Xc\ra B)$ be a \emph{gtm}-family of supermanifolds modelled on $(X, T^*_{X, -})$. Then the total space model is $(X\times B, p_X^*T^*_{X -})$. On recalling that pullbacks by morphisms commute with exterior powers, the even and odd graded components of the obstruction sheaf for this model are:
\begin{align}
\Ac_{p_X^*T^*_{X, -}}^+
\cong 
p_X^*\Ac_{T^*_{X, -}}^+\oplus p_B^*T_B
&&
\mbox{and}
&&
\Ac_{p_X^*T^*_{X, -}}^-
\cong
p_X^*\Ac_{T^*_{X, -}}^-
\label{fbiurbvibiuveuvueibcebuyce}
\end{align}
where $T_B$ is the tangent sheaf on $B$. Now the K\"unneth formula for sheaf cohomology states, for abelian sheaves $\Fc$ on $X$ and $\Ac$ on $B$, that: 
\begin{align}
H^k(X\times B, p_X^*\Fc \otimes p_B^*\Ac)
\cong 
\bigoplus_{k^\p + k^{\p\p} = k}
H^{k^\p}(X, \Fc) \otimes H^{k^{\p\p}}(B, \Ac)
\label{rhf784gf784h80f9j390}
\end{align}
In using that the sheaf cohomology of abelian sheaves is acyclic for Stein spaces (Lemma \ref{ffefrgtg5g5g5g55g5}(iii)), we can conclude the following from the K\"unneth formula:

\begin{LEM}\label{fnkbvhjbvjhdkndkln}
Let $B$ be a Stein space. Set $h^0(B,-) = \dim_\Cbb H^0(B, -)$. The obstruction space for any gtm-family of supermanifolds modelled on $(X, T^*_{X, -})$ over $B$ coincides with:
\begin{enumerate}[(i)]
	\item $h^0(B, \Oc_B)$-many copies of the obstruction space of $(X, T^*_{X, -})$ and;
	\item $h^0(B, T_B)$-many copies of $H^1(X, \Oc_X)$.
\end{enumerate}
\end{LEM}

\begin{proof}
The modelling data for any \emph{gtm}-family of supermanifolds modelled on $(X, T^*_{X, -})$ is $(X\times B, p_X^*T^*_{X, -})$. Its obstruction space is,
\begin{align*}
H^1\big(X\times B, \Ac_{p_X^*T^*_{X, -}}\big)
&=
H^1\big(X\times B, \Ac^+_{p_X^*T^*_{X, -}}\oplus \Ac^-_{p_X^*T^*_{X, -}}\big)
\\
&=
H^1\big(X\times B, \Ac^+_{p_X^*T^*_{X, -}}\big)\oplus H^1\big(X\times B, \Ac^-_{p_X^*T^*_{X, -}}\big)
\\
&\cong
H^1\big(X\times B, p_X^*\Ac_{T^*_{X, -}}^+\oplus p_B^*T_B\big)
\oplus 
H^1\big(X\times B, p_X^*\Ac^-_{T^*_{X, -}}\big)
\end{align*}
the latter isomorphism following from \eqref{fbiurbvibiuveuvueibcebuyce}. Now if $B$ is a Stein space we know that $H^k(B, -) = (0)$ for $k> 0$. Hence from the K\"unneth formula we get,
\begin{align*}
H^1\big(X\times B, p_X^*\Ac_{T^*_{X, -}}^+\oplus p_B^*T_B\big)
&\cong H^1(X\times B, p_X^*\Ac_{T^*_{X, -}}^+)\oplus H^1(X\times B, p_B^*T_B) 
\\
&\cong
\Big(H^1(X, \widetilde\Qcl_{T^*_{X, -}})\otimes H^0(B, \Oc_B)\Big)
\\
&~\oplus 
\Big(H^1(X,\Oc_X)\otimes H^0(B, T_B)\Big)
\end{align*}
and 
\begin{align*}
H^1\big(X\times B, p_X^*\Ac^-_{T^*_{X, -}}\big) \cong H^1\big(X, \Ac^-_{T^*_{X, -}}\big)
\otimes H^0(B, \Oc_B).
\end{align*}
Hence,
\begin{align*}
H^1\big(X\times B, \Ac_{p_X^*T^*_{X, -}}\big)
&\cong 
H^0(B, \Oc_B)\otimes H^1(X, \Ac_{T^*_{X, -}})
\\
&\oplus 
H^0(B, T_B)
\otimes
H^1(X,\Oc_X).
\end{align*}
The lemma now follows.
\end{proof}

\subsubsection{The Obstruction Class of a $gtm$-Family}
A crucial ingredient in our proof of Theorem \ref{rg784gf794hf98hf0j3fj34444} is in showing that the obstruction class to splitting a $gtm$-family over a Stein base $B$ will reside in the component proportional $H^0(B, \Oc_B)$ in Lemma \ref{fnkbvhjbvjhdkndkln}. More formally:

\begin{PROP}\label{rh973hf983hf80jf093fj}
Let $(\Uc\ra\Xc\ra B)$ be an analytic, $gtm$-family of supermanifolds modelled on $(X, T^*_{X, -})$ with $B$ a Stein space. Suppose the splitting type of the family is $j^\p$. Then there exists some class $\om$ in the $j^\p$-th obstruction space of $(X, T^*_{X, -})$ and a global section $s\in H^0(B, \Oc_B)$ such that the splitting type differential of the family is given by,
\[
t \longmapsto s(t)\otimes \om.
\]
The $j^\p$-th obstruction class to splitting the total space $(\Uc\ra\Xc)$ is then,
\[
\om\big(\Uc\ra\Xc\big) = s\otimes \om.
\]
\end{PROP}

\begin{proof}
As mentioned in the brief preamble to the statement of the present proposition, we need only confirm that the obstruction class to splitting $(\Uc\ra\Xc)$ will lie in the $H^0(B, \Oc_B)$ component. To see this we will use the fact that obstruction classes to splitting supermanifolds will satisfy a scaling law under the action of the multiplicative group $\Cbb^\times$. Indeed, in \cite{ONISHCLASS, DW1, BETTAQ} one can find the following result, a proof of which one can also find in \cite[Appendix A]{BETTAQ}:

\begin{LEM}\label{rfg78gf793hf98h3f03}
Fix a model $(X, T^*_{X, -})$ over $\Cbb$, i.e., that $X$ is a complex manifold and $T_{X, -}^*$ is a holomorphic vector bundle. Then the multiplicative group $\Cbb^\times$ acts on $\mbox{\emph{\v H}}^1\big(X, \Gc_{T^*_{X, -}}^{(j)}\big)$ for each $j$ and this action descends to an action on the obstruction space $H^1(X, \Ac^{(j)}_{T^*_{X, -}})$ given by, for each $\lam\in \Cbb^\times$,
\begin{align}
\om 
\stackrel{\lam\cdot}{\longmapsto}
\left\{
\begin{array}{rl}
\lam^j\om&\mbox{if $j$ is even}
\\
\lam^{j-1}\om& \mbox{if $j$ is odd}.
\end{array}
\right.
\label{rb783gf87g3f938hf9h3}
\end{align}
\qed
\end{LEM}

\noindent
Hence by Lemma \ref{rfg78gf793hf98h3f03} above, if $\om$ is the obstruction class to splitting some supermanifold, it will scale according to \eqref{rb783gf87g3f938hf9h3} under the action of $\Cbb^\times$. Now if $(\Uc\ra\Xc\ra B)$ is a family of supermanifolds modelled on $(X, T^*_{X, -})$ we have the compatibility diagram of obstruction spaces, i.e., \eqref{fvnkbvhbrybvunienvimepo}. The maps in the diagram are $\Cbb^\times$-equivariant and hence the $\Cbb^\times$-action on the obstruction space of the total space model $(Y, T^*_{Y -})$ is compatible with the $\Cbb^\times$-action on the obstruction space of the fiber model $(X, T^*_{X, -})$. More explicitly, if $j^\p$ is the splitting type of the family; its obstruction to splitting is $\om_{fam.}$; and the obstruction to splitting the fiber over some point in the base is $\om_{fib.}$, we then have:\footnote{\label{djbcevbuivbebvbebvebveo}for succinctness we are using that the action in \eqref{rb783gf87g3f938hf9h3} can be represented in the more compact form: $\om \mapsto \lam\cdot \om = \lam^{j^\p - \frac{1}{2}(1 - (-1)^{j^\p})}\om$}
\begin{align}
\lam^{j^\p - \frac{1}{2}(1 - (-1)^{j^\p})} i_*(\om_{fib.})
&=
i_*(\lam\cdot \om_{fib.})
\notag
\\
&= p_*(\lam\cdot \om_{fam.}) = \lam^{j^\p - \frac{1}{2}(1 - (-1)^{j^\p})} p_*(\om_{fam.})
 \label{fjbkrbruivbubvnvienvoeovne}
\end{align}
for all $\lam\in \Cbb^\times$ and where $i_*$ and $p_*$ denote the lower-horizontal and right-vertical maps in the compatibility diagram in \eqref{fvnkbvhbrybvunienvimepo}, respectively. Importantly, we see in \eqref{fjbkrbruivbubvnvienvoeovne} how the $\Cbb^\times$ action depends on the $\Zbb$-grading on the respective obstruction spaces. Returning now to the decomposition of the obstruction space of a $gtm$-family in Lemma \ref{fnkbvhjbvjhdkndkln}, observe that the summand in Lemma \ref{fnkbvhjbvjhdkndkln}(ii) will not come equipped with any $\Zbb$-grading, compatible with that on the obstruction spaces of the total space model and the fiber model respectively. In particular, $\Cbb^\times$ will act trivially on elements in the summand in Lemma \ref{fnkbvhjbvjhdkndkln}(ii). The present proposition now follows.
\end{proof}

\noindent
We continue our proof of Theorem \ref{rg784gf794hf98hf0j3fj34444} below.

\subsubsection{Continuation of Proof of Theorem \ref{rg784gf794hf98hf0j3fj34444}}
Let $(\Uc\ra\Xc\ra B)$ be an analytic, $gtm$-family of supermanifolds modelled on $(X, T^*_{X, -})$ over a Stein base $B$ and of splitting type $j^\p$. By Proposition \ref{rh973hf983hf80jf093fj} we know that the $j^\p$-th obstruction to splitting the total space $(\Uc\ra\Xc)$ can be written as the tensor product $s\otimes \om$, where $\om\in H^1\big(X, \Ac_{T^*_{X,-}}^{(j^\p)}\big)$ and $s \in H^0(B,\Oc_B)$ is a global section. Then on the zero locus of $s$ observe that the $j^\p$-th obstruction to splitting $(\Uc\ra\Xc)$ will vanish. And hence that the $(j^\p+1)$-th obstruction to splitting $(\Uc\ra\Xc)$ will exist. Regarding this class we have:

\begin{LEM}\label{rh79fh98hf803j90j30fj309}
Over the locus $(s = 0)$ the $(j^\p+1)$-th obstruction to splitting must vanish.
\end{LEM}

\noindent
The proof of Lemma \ref{rh79fh98hf803j90j30fj309} rests on the following more general observation.

\begin{LEM}\label{rfg78gf79f9h8fh308f830}
Let $B$ be a connected, topological space and fix vector spaces $V$ and $W$. Consider a function $\Phi: B\ra V\times W$ with components $\Phi= (\Phi^1, \Phi^2)$ that satisfy the logical constraint:
\begin{align}
\Phi^1(b) = 0 \iff \Phi^2(b) \neq0
\label{rbuieufh03jf09jfeeee}
\end{align}
for all $b\in B$. Suppose there exists $b$ such that $\Phi^1(b)\neq0$. 
If $\Phi$ is continuous, then  $\Phi^2$ vanishes on all of $B$. 
\end{LEM}

\begin{proof}
If $\Phi$ is continuous then its components $\Phi^i$ will be continuous. Now by definition, the preimage of any open set under a continuous function will be open. In defining  $B^i = \{ b\in B \mid \Phi^i(b)\neq 0\}$ for $i = 1, 2$ see that
\begin{align*}
B^1 = (\Phi^1)^{-1}\big(V\setminus 0\big)
&&\mbox{and}&&
B^2 = (\Phi^2)^{-1}\big(W\setminus 0\big)
\end{align*}
Hence $B^i\subset B$ is open for $i = 1, 2$. Now for $b\in B$ observe that either $\Phi^1(b) = 0$ or $\Phi^2(b) = 0$ but not both. This is a consequence on \eqref{rbuieufh03jf09jfeeee} and it implies therefore that $B = B^1\sqcup B^2$. But now we have a decomposition of a connected topological space into two open sets, which violates connectivity unless either $B^1$ or $B^2$ is empty. Assuming $B^1\neq\eset$ we conclude that $B^2 = \eset$, i.e., that $\Phi^2(b) = 0$ for all $b$.
\end{proof}

\noindent
\emph{Proof of Lemma $\ref{rh79fh98hf803j90j30fj309}$}.
If $(\Uc\ra\Xc\ra B)$ is an analytic family, then its splitting type differential will be analytic and hence continuous. Now assuming this family is a $gtm$-family of splitting type $j^\p$ over a Stein base $B$ the splitting type differential will be given by $t \mapsto s(t)\otimes \om$. Let $Z = (s = 0)\subset B$. Since $B$ is Stein, any subspace of $B$ will be Stein by Lemma \ref{ffefrgtg5g5g5g55g5}(ii), so $Z$ is Stein. The restriction $(\Uc\ra\Xc\ra B)|_Z$ then defines a $gtm$-family of splitting type $j^{\p\p}$ over a Stein base $Z$, where $j^{\p\p} > j^\p$. Hence \emph{it's} splitting type differential will be given by $t \mapsto s^\p(t)\otimes \om^\p$ for some $\om^\p\in H^1\big(X, \Ac_{T^*_{X, -}}^{(j^{\p\p})}\big)$ and global section $s^\p\in H^0(Z, \Oc_Z)$. Since the $j^\p$-th and $j^{\p\p}$-th obstructions to splitting a supermanifold $(\Uc\ra\Xc)$ cannot both be nonzero over the same locus of points, we see that the splitting type differential of the original family $(\Uc\ra\Xc\ra B)$ will a function as described in Lemma \ref{rfg78gf79f9h8fh308f830}, satisfying a constraint as in \eqref{rbuieufh03jf09jfeeee}. If we now assume $B$ is connected, we can apply Lemma \ref{rfg78gf79f9h8fh308f830} and conclude from analyticity (and hence, continuity) of the splitting type differential that the $j^{\p\p}$-th obstruction to splitting must necessarily vanish (presuming of course that the $j^\p$-th obstruction to splitting does not vanish identically). 
\qed
\\

\noindent
As a consequence of Lemma \ref{rh79fh98hf803j90j30fj309} we see that the fibers of $(\Uc\ra\Xc\ra B)$ over the zero locus $Z = (s = 0)$ will be split models;  and on the complement of $Z$ in $B$ that the splitting type will be constant. This concludes the proof of Theorem \ref{rg784gf794hf98hf0j3fj34444}.
\qed

\subsection{Applications}

\subsubsection{The Characteristic Section and Central Splitness}
In Proposition \ref{rh973hf983hf80jf093fj} we see that associated to any analytic, $gtm$-family over a Stein base $B$ is a global section over $B$. Note, we do not need to assume $B$ is connected here. This section is termed in what follows.

\begin{DEF}\label{rfgf794hf984hf84hf08h40}
\emph{Let $(\Uc\ra\Xc\ra B)$ be an analytic, $gtm$-family over a Stein base $B$. The global section $s\in H^0(B, \Oc_B)$ from Proposition \ref{rh973hf983hf80jf093fj} will be referred to as the \emph{characteristic section} associated to the family.}
\end{DEF}

\noindent 
A consequence of Theorem \ref{rg784gf794hf98hf0j3fj34444} is then the following.

\begin{PROP}\label{rviehv89h03j9jv09393}
Let $(\Uc\ra\Xc\ra B)$ be an analytic, $gtm$-family over a connected, Stein base $B$. If the characteristic section of this family has a zero, then the family will be weakly centrally split.
\end{PROP}

\begin{proof}
Recall from Definition \ref{tgv784gf7h498fh30jf39j} that a family $(\Uc\ra\Xc\ra B)$ is said to be weakly centrally split if the fiber-wise classifying map $\Phi_{(\Uc\ra\Xc)}$ is a map of pointed sets. We are assuming that  $(\Uc\ra\Xc\ra B)$ is an analytic, $gtm$-family over a connected, Stein base and that its characteristic section has a zero, say $b_0\in B$. Then by Theorem \ref{rg784gf794hf98hf0j3fj34444} the fiber over $b_0$ will be the split model. And hence that $\Phi_{(\Uc\ra\Xc)}$ sends $b_0$ to the base point. If we take $b_0$ to be the base point of $B$, then $\Phi_{(\Uc\ra\Xc)}$ will be a map of pointed sets. 
\end{proof}

\subsubsection{The Fiberwise Obstruction Class}
In Theorem \ref{rg784gf794hf98hf0j3fj34444} we found that for analytic, $gtm$-families over a connected, Stein base, the fiber-wise splitting type remains essentially constant. That is, outside the zero locus of its characteristic section, the fiber-wise splitting type is constant. In what follows we will argue that this will also hold for the obstruction class to splitting the fiber itself.   

\begin{THM}\label{fbjhvvbennciuebcuecoe}
To any analytic, $gtm$-family over a connected Stein base of splitting type $j^\p$, the obstruction class to splitting the fiber of the family, outside the zero locus of its characteristic section, can be identified with some fixed class in the $j^\p$-th obstruction space of the fiber-model. 
\end{THM}

\begin{proof}
Recall the tensor-product decomposition of the obstruction to splitting the total space of an analytic, $gtm$-family $(\Uc\ra\Xc\ra B)$ over a Stein base $B$ in Proposition $\ref{rh973hf983hf80jf093fj}$,
\begin{align}
\om\big(\Uc\ra\Xc) = s\otimes \om.
\label{rfh94h894hf0jf03j0}
\end{align}
Suppose the family has splitting type $j^\p$. Our argument lies in showing that the fiber of the family over any point $t$ such that $s(t)\neq0$ can be `normalised' so that its obstruction to splitting is $\om$. Given \eqref{rfh94h894hf0jf03j0}, fix a point $t\in B$ with $s(t)\neq0$. The obstruction class to splitting the fiber over $t$ is:
\begin{align}
\om\big(\Phi_{(\Uc\ra\Xc)}(t)\big) = s(t)\om.
\label{rbyf874gf9h8f0j309fj3}
\end{align}
Now, $s(t)\in \Cbb$ is some non-zero, complex number. Recall from Lemma \ref{rfg78gf793hf98h3f03} that we have an action of $\Cbb^\times$ on the obstruction space of a model, and that comes from an action on the set $\mbox{\v H}^1\big(X, \Gc_{T^*_{X, -}}^{(j^\p)}\big)$. Now $\mbox{\v H}^1\big(X, \Gc_{T^*_{X, -}}^{(j^\p)}\big)$ comprises supermanifolds of splitting type $j^\p$. If $\star$ denotes the $\Cbb^\times$-action on this set then: 
\begin{align}
\begin{array}{l}
\mbox{given any two supermanifolds $x, x^\p\in \mbox{\v H}^1\big(X, \Gc_{T^*_{X, -}}^{(j^\p)}\big)$, if $x^\p = \lam\star x$ for}
\\
\mbox{some $\lam\in \Cbb^\times$, then $x$ and $x^\p$ are isomorphic as supermanifolds.}
\end{array}
\label{rfh78gf87hf983hf0j390fj3jf93}
\end{align}
We return now to our proof of Theorem \ref{fbjhvvbennciuebcuecoe}. Let $\lam\in \Cbb^\times$ be such that,\footnote{c.f., footnote \eqref{djbcevbuivbebvbebvebveo}} 
\[
\lam^{j^\p - \frac{1}{2}(1 - (-1)^{j^\p})} = s(t).
\]
Such a number $\lam$ will always exist since $\Cbb$ is algebraically closed. Then by Lemma \ref{rfg78gf793hf98h3f03} and \eqref{rbyf874gf9h8f0j309fj3} the action $\Phi_{(\Uc\ra\Xc)}(t) \mapsto \lam^{-1}\star \Phi_{(\Uc\ra\Xc)}(t)$ sends $s(t)\om \mapsto\om$. Hence, the obstruction class to splitting the fiber of $(\Uc\ra\Xc\ra B)$ over $t$ will be $\om$. 
\end{proof}

\begin{REM}\label{fjbvjvhjbvkeboennoe}
\emph{In Theorem \ref{fbjhvvbennciuebcuecoe} we saw that the obstruction class to splitting the fiber of an analytic, $gtm$-family over a connected, Stein base will not vary with the parameters on base. We remark here that this does \emph{not necessarily} imply that the fibers of the family are all isomorphic as supermanifolds however. Only that they are isomorphic to supermanifolds with a prescribed obstruction class to splitting.}
\end{REM}

\noindent
To preempt a forthcoming illustration we present here a definition, motivated by the above remark.

\begin{DEF}\label{ryurgfyugfiuhfueofioefoeo}
\emph{Let $(\Uc\ra\Xc\ra B)$ be a $gtm$-family over a Stein base. Such a family is said to be \emph{generically isotrivial} if, outside the zero locus of its characteristic section $s$, each fiber is isomorphic as supermanifolds, i.e., that $(\Uc\ra\Xc)_t\cong (\Uc\ra\Xc)_{t^\p}$ for all $t, t^\p$ such that $s(t)\neq0$ and $s(t^\p)\neq0$.}
\end{DEF}

\subsubsection{The Compatibility Diagram}
We conclude now with the following description resulting from Theorem \ref{fbjhvvbennciuebcuecoe}. To a family of supermanifolds we have an embedding of the fiber into the total space. This requires an embedding of models and so we have a compatibility diagram associated to this embedding (c.f., Definition \ref{rgf79f9h9hohiooihfeofoi}). When we have a $gtm$-family of supermanifolds modelled on $(X, T^*_{X, -})$ over a connected, Stein space $B$, the model for the total space is $(X\times B, p_X^*T^*_{X, -})$. Supposing it has splitting type $j^\p$ we find, over each $t\in B$ the compatibility diagram:
\begin{align}
\xymatrix{
\ar[d] H^1\big(X\times B, \Ac_{T^*_{X, -}; p_X^*T^*_{X, -}}^{(j^\p)}\big) \ar[rr] & & H^1\big(X\times B, \Ac_{p_X^*T^*_{X,-}}^{(j^\p)}\big)\ar[d]
\\
H^1\big(X, \Ac_{T^*_{X, -}}^{(j^\p)}\big)\ar[rr] & & H^1\big(X, \Ac^{(j^\p)}_{p_X^*T^*_{X,-}}|_X\big).
}
\label{jdkvbhjdbvknioionoeun}
\end{align}
Note that while the spaces involved in the above diagram are `independent of $t$', the maps between them are certainly not. We can see this quite clearly in the following illustration with which we conclude this section.

\begin{ILL}\label{h78gf87g9h38hf03f093j}
Let $(\Uc\ra\Xc\ra B)$ be an analytic, $gtm$-family over a connected, Stein base $B$ with splitting type $j^\p$. In this case we then have a projection $H^1\big(X\times B, \Ac_{p_X^*T^*_{X,-}}\big)\ra H^0(B, \Oc_B)\otimes H^1\big(X, \Ac_{T^*_{X, -}}^{(j^\p)}\big)$ from Lemma \ref{fnkbvhjbvjhdkndkln}. At \emph{any} point $t$ we then have the evaluation-at-$t$ map sending this tensor product to $H^1\big(X, \Ac_{T^*_{X, -}}^{(j^\p)}\big)$. And so, by Theorem \ref{fbjhvvbennciuebcuecoe}, we conclude that the compatibility diagram reduces to, for each $t$:
\begin{align*}
\xymatrix{
\ar[d] H^1\big(X\times B, \Ac_{T^*_{X, -}; p_X^*T^*_{X, -}}^{(j^\p)}\big) \ar[r] & H^1\big(X\times B, \Ac_{p_X^*T^*_{X,-}}^{(j^\p)}\big)\ar[r] \ar[d] & H^0(B, \Oc_B)\otimes H^1(X, \Ac_{T^*_{X, -}}^{(j^\p)})\ar[d]^{\mbox{ev}_t}
\\
H^1\big(X, \Ac_{T^*_{X, -}}^{(j^\p)}\big) \ar@/_2pc/@{-->}[rr]_{-\otimes \mbox{ev}_t(s)}  \ar[r] & H^1\big(X, \Ac^{(j^\p)}_{p_X^*T^*_{X,-}}|_X\big)\ar[r] & 
H^1\big(X, \Ac_{T^*_{X, -}}^{(j^\p)}\big)
}
\end{align*}
~\\
where $s$ is the characteristic section for the family $(\Uc\ra\Xc\ra B)$. The horizontal map $H^1\big(X, \Ac^{(j^\p)}_{p_X^*T^*_{X,-}}|_X\big)\ra H^1\big(X, \Ac_{T^*_{X, -}}^{(j^\p)}\big)$ is given by projection upon recalling the relation between $\Ac^{(j^\p)}_{p_X^*T^*_{X,-}}$ and $\Ac^{(j^\p)}_{T^*_{X,-}}$ in \eqref{fbiurbvibiuveuvueibcebuyce}.
\end{ILL}

\section{Illustration: Rothstein's Deformation}
\label{rfg78gf78gf9h38fh3f03}

\noindent 
In \cite{ROTHDEF}, Rothstein gave a construction of a one-dimensional family of complex supermanifolds. We describe this construction here. As we will see, it is an example of an analytic, $gtm$-family over a connected, Stein base $B = \Cbb$. We might view Theorem \ref{rg784gf794hf98hf0j3fj34444} and Theorem \ref{fbjhvvbennciuebcuecoe} as generalisations of Rothstein's construction. Our presentation of this construction will utilise the the viewpoint of families of `glueing data' and so we embark now on the necessary digression.

\subsection{Glueing Data for $gtm$-Families over $\Cbb^m$}
With reference to local coordinates and glueing data, Theorem \ref{rg784gf794hf98hf0j3fj34444} and Theorem \ref{fbjhvvbennciuebcuecoe} can be illustrated explicitly in the case where the connected, Stein base is $B = \Cbb^m$.

\subsubsection{Local Coordinates and Glueing}
The notion of supermanifold adopted in this article was given in Definition \ref{furbfbiuff3f03jf3jf}. In Remark \ref{fknkrnjkvbrvbvnoinio4nfo4f4} it was mentioned how supermanifolds, classically, are certain kinds of locally ringed spaces. While this viewpoint loses information about the splitting type, it caters to the more geometric understanding of supermanifolds as being glued together by model spaces, which are split models. Supermanifolds of splitting type $j^\p$ are then supermanifolds equipped with a covering by split models and transition functions of a particular form. Further details on supermanifolds via glueing can be found in \cite{BETTPHD, BETTOBSTHICK}. We give a brief description of glueing data for families of supermanifolds in what follows.
\\\\
If $(\Uc\ra\Xc\ra \Cbb^m)$ is a \emph{gtm}-family, then $(\Uc\ra\Xc)$ will be a supermanifold modelled on $(X\times \Cbb^m, p_X^*T^*_{X, -})$. Set $\Uc = (\Uc_\al)_{\al\in I}$. Supermanifolds are locally split so we can choose $\Uc$ to be a splitting cover. This means each $\Uc_\al\in \Uc$ will be split as a supermanifold in its own right. With $U_\al = (\Uc_\al)_{\red}\subset X\times \Cbb^m$ and $q = \mathrm{rank}~T^*_{X, -}$ the pullback $p_X^*T^*_{X, -}$ will be locally free and of rank $q$. Hence $\Uc_\al \cong S(U_\al, \wedge^\bt \Oc^{\oplus q}_{U_\al})$ and $(\Uc\ra \Xc)$ realises $\Xc$ as being glued by the split supermanifolds $(S(U_\al, \wedge^\bt \Oc^{\oplus q}_{U_\al}))_{\al\in I}$. Now let $(z|\q)$ denote local coordinates on $\Uc_\al$. Since $\Uc_\al$ is split, $z$ defines local coordinates on $U_\al$. Since $U_\al\subset X\times \Cbb^m$ we can write $z = (x, t)$, where $x$ are local coordinates on $U_\al\cap X$ and $t$ are global coordinates $\Cbb^m$. Hence local coordinates on $\Uc_\al$ are $(z|\q) = (x, t|\q)$. Note that since $t$ is a global coordinate on $\Cbb^m$, if local coordinates over $\Uc_\be$ are $(z^\p|\eta) = (y, t^\p|\eta)$, then over $\Uc_\al\cap\Uc_\be$ we have $t^\p = t$.

We sum up these remarks in the following.

\begin{LEM}\label{ruihfhf89hf83j0fj309j0000}
Let $(\Uc\ra \Xc\ra \Cbb^m)$ be a gtm-family of supermanifolds of splitting type $j^\p$. Suppose $\Uc = (\Uc_\al)_{\al\in I}$ is a splitting cover. Then the transition functions $\vartheta$ for the supermanifold $(\Uc\ra\Xc)$ are of the following form on intersections $\Uc_\al\cap\Uc_\be$,
\begin{align*}
y 
&= \vartheta^+_{\al\be}(x, t|\q) 
\\
&= f_{\al\be}(x) + (j^\p+1 \mod 2)\q^{j^\p} g_{\al\be}(x, t) + \Oc(\q^{j^\p+1})
\\
t^\p &= t
\\
\eta
&=
\vartheta_{\al\be}^-(x, t|\q) 
\\
&=
\zeta_{\al\be}(x)\q + (j^\p\mod 2) \q^{j^\p} h_{\al\be}(x, t) + \Oc(\q^{j^\p+1})
\end{align*}
where $\Oc(\q^{j^\p+1})$ denotes terms of order $\q^{j^\p+1}$ and higher; $g_{\al\be}$ and $h_{\al\be}$ are analytic over $U_\al\cap U_\be$.
\qed
\end{LEM}

\subsubsection{The Splitting Type Differential}
As detailed in \cite{YMAN, BETTPHD, BETTOBSTHICK}, from glueing data for supermanifolds we can readily recover cocycle expressions for the obstruction classes to splitting. And so for a $gtm$-family $(\Uc\ra \Xc\ra B)$ of splitting type $j^\p$, given by glueing data $\vartheta =(\vartheta_{\al\be})$ as in Lemma \ref{ruihfhf89hf83j0fj309j0000}, a cocycle representative for its obstruction class to splitting on $\Uc_\al\cap\Uc_\be$ is:
\begin{align*}
\om_{\al\be}\big(\Uc\ra\Xc\big)
&=
\frac{1}{j^\p!}\q\frac{\pt \vartheta_{\al\be}(z|\q)}{\pt \q}\otimes 
\left((j^\p+1\mod 2)\frac{\pt}{\pt z^\p} + (j^\p\mod 2) \frac{\pt}{\pt \eta}
\right)
\notag
\\
&=
(j^\p+1\mod 2) g_{\al\be}(x, t)\q^{j^\p}~\frac{\pt}{\pt y}
+
(j^\p\mod 2)
h_{\al\be}(x, t)\q^{j^\p}~\frac{\pt}{\pt \eta}
\end{align*}
where recall that $z = (x, t)$ and $z^\p = (y, t)$. Evidently, the splitting type differential is given by 
\begin{align}
t
\longmapsto 
\left[\left\{
(j^\p+1\mod 2) g_{\al\be}(x, t)\q^{j^\p}~\frac{\pt}{\pt y}
+
(j^\p\mod 2)
h_{\al\be}(x, t)\q^{j^\p}~\frac{\pt}{\pt \eta}
\right\}
\right].
\label{hf784f97hf983fj393jddd}
\end{align}
If $(\Uc\ra\Xc\ra B)$ is an \emph{analytic family} in the sense of Definition \ref{rfh489f98hfj3039k} then the mapping in \eqref{hf784f97hf983fj393jddd} is analytic. This amounts to either $g_{\al\be}(x, t)$ or $h_{\al\be}(x, t)$ being analytic in $t$ over the intersection $U_\al\cap U_\be\cap \Cbb^m$.

\subsubsection{Illustration of Theorem \ref{rg784gf794hf98hf0j3fj34444} and Theorem \ref{fbjhvvbennciuebcuecoe}} From Lemma \ref{ruihfhf89hf83j0fj309j0000} we see that $\pt t^\p/\pt \q = 0$ and so there is no $\pt/\pt t$-component in \eqref{hf784f97hf983fj393jddd}. Hence the obstruction to splitting $(\Uc\ra\Xc)$ will lie in the Lemma \ref{fnkbvhjbvjhdkndkln}\emph{(i)}-component, as expected from Proposition \ref{rh973hf983hf80jf093fj}. This proposition  moreover asserts that the coefficient functions $g_{\al\be}(x, t)$ and $h_{\al\be}(x, t)$ in \eqref{hf784f97hf983fj393jddd} can be written: $g_{\al\be}(x, t) = s(t)g^\p_{\al\be}(x)$ and $h_{\al\be}(x, t) = s(t)h^\p_{\al\be}(x)$ over $\Uc_\al\cap \Uc_\be$ and some global section $s\in H^0(B, \Oc_B)$, referred to as the `characteristic section' in Definition \ref{rfgf794hf984hf84hf08h40}. Theorem \ref{rg784gf794hf98hf0j3fj34444} then asserts that the higher order components of the transition data in Lemma \ref{ruihfhf89hf83j0fj309j0000}, i.e., the terms in $\Oc(\q^{j^\p+1})$, will be proportional to $s(t)$. In particular, they will vanish when $s$ vanishes. Finally, Theorem \ref{fbjhvvbennciuebcuecoe} asserts that the obstruction class to splitting the fiber over any point $t\in \Cbb^m$ such that $s(t)\neq0$ can be identified with the class:
\[
\left[\left\{
(j^\p+1\mod 2) g_{\al\be}^\p(x)\q^{j^\p}~\frac{\pt}{\pt y}
+
(j^\p\mod 2)
h_{\al\be}^\p(x)\q^{j^\p}~\frac{\pt}{\pt \eta}
\right\}
\right]
\in
H^1\big(X, \Ac^{(j^\p)}_{T^*_{X, -}}\big)
\]
where $(X, T^*_{X, -})$ is the fiber model of the family $(\Uc\ra\Xc\ra\Cbb^m)$. We now present a construction of a family of generically isotrivial supermanifolds, following Rothstein in \cite{ROTHDEF}.

\subsection{Rothstein's Construction}
Fix a supermanifold $(\Ufr\ra \Xfr)$ of splitting type $j$. We have transition functions $\rho$ which, over any non-empty intersection $U_\al\cap U_\be$, is written:
\begin{align}
\rho_{\al\be}(x|\q) = f_{\al\be}(x) + \zeta_{\al\be}(x)\q + h_{\al\be}(x)\q^j + \Oc(\q^{j+1})
\label{rhf894hf98h8f30f3}
\end{align}
where $\zeta_{\al\be}(x)\q = \sum_{q =1}^q\zeta_{\al\be}^a(x)\q_a$; $h_{\al\be}(x)\q^j = \sum_{|I|=j} h^I_{\al\be}(x)\q_I$ for $I$ a multi-index; and where $\Oc(\q^{j+1})$ denotes terms proportional to $\q^\ell$ for $\ell>j$. Note that $\rho$ in \eqref{rhf894hf98h8f30f3} ought to be thought of as glueing the open set $\Ufr_{\al\be}\subset \Ufr_\al$ to $\Ufr_{\be\al}\subset \Ufr_\be$. The variables $(x|\q)$ define an even resp. odd coordinate system on $\Ufr_\al$. Now for each $t\in \Cbb$ consider the map,
\begin{align*}
x\mapsto x&&\q\mapsto \q&&\mbox{and}&&\q^j \mapsto t^j\q^j.
\end{align*}
This defines an automorphism of $\Uc_\al = \Ufr_\al\times \Cbb$. Consider transition functions modifying those in \eqref{rhf894hf98h8f30f3}:
\begin{align}
\vartheta_{\al\be}(x, t|\q)
=
f_{\al\be}(x) + \zeta_{\al\be}(x)\q + t^jh_{\al\be}(x)\q^j + \Oc(t^{j+1}\q^{j+1})
\label{rgf874gf87hf98h3083j03}
\end{align}
thereby glueing $\Ufr_{\al\be}\times \Cbb\subset \Uc_\al$ to $\Ufr_{\be\al}\times \Cbb \subset \Uc_\be$. That the cocycle condition for $\vartheta = (\rho_{\al\be}(x, t|\q))_{\al,\be}$ will be satisfied on triple intersections $\Ufr_\al\cap\Ufr_\be\cap\Ufr_\gam$ will follow from the fact that the cocycle condition will be satisfied by $\rho$ in \eqref{rhf894hf98h8f30f3} by assumption. Hence $\big(\Uc, \vartheta =  (\rho_{\al\be}(x, t|\q))_{\al,\be}\big)$ will glue to define a supermanifold $\Xc$. The projection $\Ufr\ra \Cbb$ given by $(x, t|\q) \mapsto t$ is globally well defined and defines a family $(\Uc \ra\Xc\ra \Cbb)$, where $\Uc = (\Uc_\al)$ and $\Uc_\al = \Ufr_\al\times \Cbb$.  Clearly $\Xc_\red = X\times \Cbb$ and $T_{\Xc_\red, -}^* = p_X^*T_{X, -}^*$ is the locally free sheaf with transition functions $(\zeta_{\al\be})$. Hence $(\Uc \ra\Xc\ra \Cbb)$ is a $gtm$-family. Furthermore, from \eqref{rgf874gf87hf98h3083j03} we see that the splitting type differential of the family is:
\[
\Phi_{(\Uc\ra\Xc)}(t) = t^j\om
\]
where $\om$ is the obstruction to splitting $(\Ufr\ra\Xfr)$. From the above expression we see that $\Phi_{(\Uc\ra\Xc)}$ is analytic and so $(\Uc \ra\Xc\ra \Cbb)$ is an analytic, $gtm$-family over the connected, Stein base $B =\Cbb$. In writing $\Cbb = \mathrm{Spec}~\Cbb[t]$, for $t$ a global coordinate on $\Cbb$, the characteristic section for $(\Uc\ra\Xc\ra \Cbb)$ is given by $s: t\mapsto t^j$. Furthermore, since $t = 0$ is a zero of $s$, this family will be weakly centrally split over the pointed base $(\Cbb, 0)$ by Proposition \ref{rviehv89h03j9jv09393}. We can also directly see this in \eqref{rgf874gf87hf98h3083j03} upon setting $t =0$. Moreover, our original supermanifold $(\Ufr\ra\Xfr)$ is recovered from $(\Uc\ra\Xc\ra \Cbb)$ as the fiber over $t= 1$.

\begin{DEF}\label{rgf874g7fh9f8h3fhf803}
\emph{The analytic, $gtm$-family $(\Uc\ra\Xc\ra \Cbb)$ constructed above from a given supermanifold $(\Ufr\ra\Xfr)$ will be referred to as the \emph{Rothstein family associated to $(\Ufr\ra\Xfr)$.}}
\end{DEF}

\begin{REM}
\emph{Rothstein in \cite{ROTHDEF} referred to the `Rothstein family' (in the sense of Definition \ref{rgf874g7fh9f8h3fhf803}) as a deformation of the split model. We observe that this description is compatible with the definition of `splitting type deformations' given in Definition \ref{fvuirviurvieiojowjwo}. That is, the Rothstein family of a supermanifold might also be viewed as \emph a splitting type deformation of the split model.}
\end{REM}

\subsection{Generic Isotriviality}
Implicit in Rothstein's work in \cite{ROTHDEF} is that the Rothstein family of a supermanifold constructed above will be generically isotrivial, in the sense of Definition \ref{ryurgfyugfiuhfueofioefoeo}. We present an argument of this below.

\begin{THM}\label{rbfiugf97hf98h30f093f03}
The Rothstein family of any supermanifold is generically isotrivial.
\end{THM} 

\begin{proof}
In Lemma \ref{rfg78gf793hf98h3f03} and \eqref{rfh78gf87hf983hf0j390fj3jf93}, we established that the multiplicative group $\Cbb^\times$ acts as automorphisms of supermanifolds with prescribed splitting type. Now fix a supermanifold $(\Ufr\ra\Xfr)$ of splitting type $j$ and suppose its transitions functions are given by \eqref{rhf894hf98h8f30f3}. Write $\rho = (\rho^+| \rho^-)$, where the subscripts indicate that only monomials in the odd variables of even $`+'$, resp. odd `$-$' degree are included. From Lemma \ref{ruihfhf89hf83j0fj309j0000} we see that $\rho^+$ resp. $\rho^-$ glue the even resp. odd coordinates, i.e., on $U_\al\cap U_\be$ that $(y|\eta) = (\rho_{\al\be}^+(x|\q)| \rho_{\al\be}^-(x|\q))$. On supposing $(\Ufr\ra\Xfr)$ is modelled on $(X, T^*_{X, -})$, the transition functions $\rho =(\rho^+|\rho^-)$ define a class in $\mbox{\v H}^1\big(X, \Gc_{T^*_{X, -}}^{(j)}\big)$. The action by $\Cbb^\times$, denoted $\star$, is given by:
 \begin{align}
(t\star \rho)_{\al\be}(x|\q)
=
\left(\begin{array}{l}
\rho^+_{\al\be}(x|t\q)
\\
t\rho^-_{\al\be}(x|t\q)
\end{array}
\right).
\label{fnbvuirbviunvoenioeop}
\end{align}
On expanding the right-hand side of \eqref{fnbvuirbviunvoenioeop}, we will recover the glueing data in \eqref{rgf874gf87hf98h3083j03}. This allows us to conclude that the fiber of the Rothstein family $(\Uc\ra\Xc\ra \Cbb)$ associated to $(\Ufr\ra\Xfr)$ over $t$ is simply $t\star (\Ufr\ra\Xfr)$. Hence, outside the zero locus of the characteristic section, i.e., for $t\neq0$, the fibers will all be isomorphic to $(\Ufr\ra\Xfr)$ (c.f., \eqref{rfh78gf87hf983hf0j390fj3jf93}). As such the Rothstein family will be generically isotrivial. 
\end{proof}

\subsection{Glueing the Rothstein Family over $\Pbb^1_\Cbb$}

\subsubsection{Preliminaries}
Let $B$ be a space with covering $(B_i)_{i\in I}$. If we are given a collection of families over $B_i$, say $(Z_i\ra B_i)_{i\in I}$, it a natural question to ask when they might glue to define some space $(Z\ra B)$. The resolution to such questions motivate the theory of stacks and descent, a reference to which is \cite{VIST}. For our purposes we will only need the following preliminary result which shows that the category of continuous functions forms a stack over topological spaces. From \cite[Proposition 4.1, pp. 70-1]{VIST} we have:

\begin{PROP}\label{fjbjvurbvuniovneinve}
Fix a space $B$ and a covering of $B$ by spaces $(B_i)_{i\in I}$ and let $(Z_i\ra B_i)_{i\in I}$ be a collection of families over $B_i$. Then there will exist a space $(Z\ra B)$ over $B$ and an isomorphism $(Z\ra B)|_{U_i} \cong (Z_i\ra B_i)$ for each $i$ if the following conditions hold:
\begin{enumerate}[(i)]
	\item on each non-empty intersection $B_i\cap B_j$ there exists an isomorphism $\vartheta_{ij} : Z_i|_{B_i\cap B_j} \stackrel{\cong}{\ra} Z_j|_{B_i\cap B_j}$ over $B_i\cap B_j$ with $\vartheta_{ii} = {\bf 1}$ for all $i$ and;
	\item on every non-empty, triple intersection $B_i\cap B_j \cap B_k$ we have a commutative diagram over $B_i\cap B_j\cap B_k$,
	\[
	\xymatrix{
	\ar[dr]_{\vartheta_{ik}} Z_i|_{B_i\cap B_j \cap B_k} \ar[rr]^{\vartheta_{ij}} & & Z_j|_{B_i\cap B_j \cap B_k}\ar[dl]^{\vartheta_{jk}}
	\\
	& Z_k|_{B_i\cap B_j \cap B_k}
	}
	\]
\end{enumerate}
\qed
\end{PROP}

\begin{REM}\label{fjbcjevcyuevcibeocneo}
\emph{In the case where the base $B$ in Proposition \ref{fjbjvurbvuniovneinve} is one-dimensional there will exist coverings with no non-trivial, triple intersections. Over such coverings, only Proposition \ref{fjbjvurbvuniovneinve}\emph{(i)} will be relevant to check to ensure the existence of a space over $B$ from the datum of local families $(Z_i\ra B_i)_{i\in I}$.}
\end{REM}

\subsubsection{Glueing the Rothstein Family} Following Proposition \ref{fjbjvurbvuniovneinve} we give the following definition.

\begin{DEF}\label{rfh479gh94hg984hg80gh044}
\emph{Let $B$ be given with an open covering $(B_i)_{i\in I}$ and suppose we have a collection of families $(X_i\ra B_i)_{i\in I}$ of spaces or superspaces. If there exist functions $(\vartheta_{ij})_{i, j\in I}$ satisfying Proposition \ref{fjbjvurbvuniovneinve}\emph{(i)} and \ref{fjbjvurbvuniovneinve}\emph{(ii)}, then we say $(X_i\ra B_i)_{i\in I}$ will \emph{glue to define a space over $B$}.}
\end{DEF}

\noindent
Our objective is now to prove the following:

\begin{THM}\label{rf8g874g974h98fh8fj309j}
Let $B = \Pbb^1_\Cbb$ be equipped with its standard covering $\{U_1, U_2\}$. Then for any supermanifold $x = (\Ufr\ra\Xfr)$, its Rothstein family $(R_i(x)\ra U_i)_{i = 1, 2}$ will glue to define a space over $\Pbb^1_\Cbb$.
\end{THM}

\begin{proof}
The standard covering of $\Pbb^1_\Cbb$ is by open sets $U_1\cup U_2$, where $U_1 = \mathrm{Spec}~\Cbb[t]$, $U_2 = \mathrm{Spec}~\Cbb[t^\p]$ and transition data $t\mapsto t^\p = 1/t$. There are no non-empty, triple intersections so, by Remark \ref{fjbcjevcyuevcibeocneo} and Definition \ref{rfh479gh94hg984hg80gh044}, we only need to confirm Proposition \ref{fjbjvurbvuniovneinve}\emph{(i)}. Fix a supermanifold $x = (\Ufr\ra\Xfr)$. Then over each open set  we can form the Rothstein family associated to $x$ which we denote respectively by $R_1(x) = (\Uc^\p\ra\Xc^\p\ra U_1)$ and $R_2(x) = (\Uc^{\p\p}\ra\Xc^{\p\p}\ra U_2)$. On the intersection $U_1\cap U_2$ we see by generic isotriviality in Theorem \ref{rbfiugf97hf98h30f093f03} that $t\mapsto t^{-2}t = 1/t = t^\p$ will lift to an isomorphism of supermanifolds,
\[
\big(R_1(x)\big)_t
\cong 
\big(R_2(x)\big)_{t^\p} =
\big(t^{-2}\big)\star
 \big(R_1(x)\big)_t.
\]
Hence we have an isomorphism $\vartheta_{12}: R_1(x)|_{U_1\cap U_2}\cong R_2(x)|_{U_1\cap U_2}$ over $U_1\cap U_2$, thereby confirming Proposition \ref{fjbjvurbvuniovneinve}\emph{(i)}. 
\end{proof}

\part{Obstruction Maps}
\setcounter{section}{0}

\section{$gtm$-Families over Superspaces: Preliminary Constructions}
\label{fjbvhvvyuviruriormprmopvmrp}

\subsection{Definitions}

\subsubsection{The $gt$-Model}
To an isotrovial family of complex manifolds $Y\ra B$ with typical fiber $X$ and a family of sheaves $(T^*_{Y, -})_{/T^*_{X, -}}$, we formed the notion of a \emph{globally trivial model} in Definition \ref{rgf784gf398fh30}(iii). This is generalised below to accomodate superspace bases.

\begin{DEF}\label{rfh74gf9hf830fj93jf}
\emph{Let $(Y, T^*_{Y, -})$ be a model and suppose:
\begin{enumerate}[(i)]
	\item $Y= X\times B$ for complex manifolds $X$ and $B$ and;
	\item  that $T^*_{Y, -}$ fits into a short exact sequence
	\[
	0\lra p_B^*T_{B, -}^* \lra T_{Y, -}^* \lra p_X^*T^*_{X, -}\lra0
	\]
	where $T_{B, -}^*$ resp. $T_{X, -}^*$ are locally free sheaves on $B$ resp. $X$.
\end{enumerate}
Then we say the model $(Y, T^*_{Y, -})$ is of \emph{gt}-type \emph{over $(B, T^*_{B, -})$ and with fiber $(X, T^*_{X, -})$}. If the latter two models are understood, then $(Y, T^*_{Y, -})$ will more simply be referred to as being \emph{of $gt$-type} or a \emph{$gt$-model}.}
\end{DEF}

\noindent
From Definition \ref{rfh74gf9hf830fj93jf}(ii) observe that the odd cotangent bundle $T^*_{Y, -}$ of $gt$-model will define an extension class $\Theta(T_{Y, -}^*)\in \mathrm{Ext}_{\Oc_Y}^1\big(p_X^*T^*_{X, -}, p_B^*T^*_{B, -}\big)$.\footnote{We will usually omit reference to the projection maps $p_X$ and $p_B$ for sake of readability.} As we will see, this class will play an important role in the characterisation of $gtm$-families. 

\begin{DEF}
\emph{Let $(Y, T^*_{Y, -})$ be a $gt$-model. The extension class $\Theta(T^*_{Y, -})$ will be referred to as its \emph{class} or abstractly as the \emph{model class}.}
\end{DEF}

\subsubsection{The $gtm$-Family}
Given a family of (super)manifolds $\Xc\ra B$, the inclusion of a point $b: \mathrm{Spec}~\Cbb\ra B$ in the base lies under the inclusion of the fiber over the point, i.e., that we have a diagram of spaces
\[
\xymatrix{
\Xc_b \ar[d]\ar[r] & \Xc\ar[d]
\\
\mathrm{Spec}~\Cbb\ar[r]^b& B
}
\] 
And so the fiber $\Xc_b$ might be viewed as the pullback of $\Xc\ra B$ along $b: \mathrm{Spec}~\Cbb\ra B$. If we consider a superspace base $\widetilde B$ modelled on $(B, T^*_{B, -})$, then we have a natural embedding $i: B\subset \widetilde B$. This embedding forms the basis for the following definition of a $gtm$-family over a superspace base.

\begin{DEF}\label{ckjvbbvjfbvhjbnckld}
\emph{Fix models $(X, T^*_{X, -})$ and $(B, T^*_{B, -})$ and suppose $\widetilde B$ is a supermanifold modelled on $(B, T^*_{B, -})$. A morphism of supermanifolds $(\widetilde \Uc\ra\widetilde\Xc \ra\widetilde B)$ is said to be a \emph{$gtm$-family of supermanifolds modelled on $(X, T^*_{X, -})$ over $\widetilde B$} if:
\begin{enumerate}[(i)]
	\item the model for total space $(\widetilde\Uc\ra\widetilde\Xc)$ is a $gt$-model over $(B, T^*_{B, -})$ with fiber $(X, T^*_{X, -})$ and;
	\item there exists a $gtm$-family of supermanifolds modelled on $(X, T^*_{X, -})$ over $B$, say $(\Uc\ra\Xc\ra B)$, in $(\widetilde\Uc\ra\widetilde\Xc)$ which lies over the natural embedding $B\subset \widetilde B$, i.e., giving rise to a commutative diagram of spaces
	\[
	\xymatrix{
	\Uc \ar[r] \ar[d] & \Xc \ar[r] \ar[d] & B \ar[d]
	\\
	\widetilde \Uc\ar[r] & \widetilde\Xc \ar[r] & \widetilde B.
	}
	\]
\end{enumerate}
The family $(\Uc\ra\Xc\ra B)$ will be referred to as the \emph{underlying family} of $(\widetilde\Uc\ra\widetilde\Xc\ra\widetilde B)$. The family  $(\widetilde\Uc\ra\widetilde\Xc\ra\widetilde B)$ is analytic iff its underlying family is analytic in the sense of Definition \ref{rfh489f98hfj3039k}.}
\end{DEF}

\noindent
Among the main examples in the literature motivating the definitions given so far is that of super Riemann surfaces and their deformations, discussed below.

\begin{EX}\label{reugf78hf98h30f830}
In \cite{BETTSRS}, deformations of super Riemann surfaces over the superspace base $\Cbb^{0|n}$ were studied. If $S(C, T^*_{C -})$ denotes a super Riemann surface and $(\Uc\ra\Xc\ra\Cbb^{0|n})$ a deformation of it, then the total space $(\Uc\ra\Xc)$ will be modelled on $(C, T_{\Xc, -}^*)$, where $T_{\Xc, -}^*\in \mathrm{Ext}^1(T_{C, -}^*, \oplus^n\Oc_C)$. When $n = 1$ it was shown in \cite{BETTPHD, BETTSRS}, following Donagi and Witten in \cite{DW1}, that the class of this model coincides with the obstruction to splitting the total space of the deformation. Now, the reduced space of $\Cbb^{0|n}$ is a point and we have the inclusion $\mathrm{Spec}~\Cbb\ra \Cbb^{0|n}$. Hence the family underlying $(\Uc\ra\Xc\ra\Cbb^{0|n})$ is the original super Riemann surface $S(C, T^*_{C, -})$. Evidently, the deformations of super Riemann surfaces in \cite{BETTSRS} are examples of analytic, $gtm$-families over superspace bases.
\end{EX}

\begin{REM}
\emph{The notion of a $gtm$-family over a superspace being `centrally' or `weakly centrally' split can be adapted from Definitions \ref{rhf94hf98h80fh309fj0} and \ref{tgv784gf7h498fh30jf39j}, which applied to families over reduced bases. If we declare a $gtm$-family over a superpsace to be centrally or weakly centrally split if its underlying family is, then Example \ref{fnvrbvuirnvionieovoemv} will indeed be an example of a weakly centrally split family in this sense.}
\end{REM}

\subsection{Projectability}  
In this article we have restricted our study to families of supermanifolds over a reduced base $B$. If we consider now a projectable base $\widetilde B$, then $\widetilde B$ is a supermanifold which admits a projection $\pi : \widetilde B\ra B$, where $B = \widetilde B_\red$ is the reduced space. There may exist many choices of projection and so $\widetilde B$ is said to be \emph{projected} if it is endowed with one such choice $\pi$. Accordingly, projected superspaces $\widetilde B$ will be denoted $(\widetilde B, \pi)$. If we are given a family of supermanifolds $(\widetilde\Uc\ra \widetilde\Xc \ra \widetilde B)$, it is natural ask whether composition with a projection $\pi: \widetilde B \ra B$ will realise $(\widetilde\Uc\ra\widetilde\Xc)$ as a family of supermanifolds over a reduced base $B$. Now note that we already have family of supermanifolds over $B$ associated to $(\widetilde\Uc\ra \widetilde\Xc \ra \widetilde B)$, referred to in Definition \ref{ckjvbbvjfbvhjbnckld} as the `underlying family'. This leads to the following definition.

\begin{DEF}
\emph{Let $(\widetilde\Uc\ra \widetilde\Xc \ra \widetilde B)$ be a $gtm$-family of supermanifolds, with $\widetilde B$ a projectable superspace. This family is itself said to be \emph{projectable} if there exists a projection $\pi : \widetilde B\ra B$ which lifts to a projection $\widetilde\pi: \widetilde \Xc \ra \Xc$, where $\Xc$ is the total space of the underlying family.}
\end{DEF}

\noindent
We continue our discussion of super Riemann surfaces from Example \ref{reugf78hf98h30f830} in the following, suggesting that families over projectable superspaces need not be projectable. 

\begin{EX}\label{rgf784gf73hf83h0j9f33333}
Let $(\Uc\ra\Xc\ra\Cbb^{0|n})$ be a deformation of a super Riemann surface $S(C, T^*_{C, -})$. We saw in Example \ref{reugf78hf98h30f830} that it is a $gtm$-family of supermanifolds over the superspace base $\Cbb^{0|n}$. When $n = 1$ it was mentioned that the model class of this family will coincide with the obstruction to splitting the total space $(\Uc\ra\Xc)$. Since $\Xc$ is a $(1|2)$-dimensional supermanifold with reduced space $C$, its obstruction to splitting will coincide with its obstruction to projection onto $C$. Now suppose $(\Uc\ra\Xc\ra\Cbb^{0|n})$ is projectable. Then, as is clear from the discussion in Example \ref{reugf78hf98h30f830}, we will have a projection $\Xc \ra S(C, T^*_{C, -})$ and hence, by composition, a projection $\Xc \ra S(C, T^*_{C, -})\ra C$, revealing that $\Xc$ will itself be projectable as a superspace. If the model class of $(\Uc\ra \Xc\ra \Cbb^{0|1})$ is non-trivial however, we know that the total space will not be projectable. Therefore, to conclude, if the model class of this family is non-trivial, the family will not be projectable. 
\end{EX}

\subsection{The Compatibility Diagram}
We turn now to the problem of comparing the obstruction to splitting the total space of a $gtm$-family over a superspace base with that of the fiber. That is, in formulating the compatibility diagram. 

\subsubsection{Obstruction Spaces}
Recall by our definition that a $gtm$-family of supermanifolds over a superspace base $(\widetilde \Uc\ra\widetilde\Xc\ra\widetilde B)$ contains a $gtm$-family over the reduced base $(\Uc\ra\Xc\ra B)$.  Accordingly, we have an embedding of supermanifolds from Definition \ref{ckjvbbvjfbvhjbnckld}(ii), reproduced below for convenience:
\[
\xymatrix{
\Uc \ar[r] \ar[d] & \Xc \ar[r] \ar[d] & B \ar[d]
\\
\widetilde \Uc\ar[r] & \widetilde\Xc \ar[r] & \widetilde B.
}
\]
The $gt$-model for $(\Uc\ra\Xc\ra B)$ is $(Y, p_X^*T^*_{X, -})$ whereas the $gt$-model for $(\widetilde \Uc\ra\widetilde\Xc\ra\widetilde B)$ is $(Y, T^*_{Y, -})$, where $Y = X\times B$. Since we have an embedding of supermanifolds $\Xc\subset \widetilde \Xc$, we therefore have a relation between \emph{their} respective obstruction classes to splitting. From \eqref{fbiurbvibiuveuvueibcebuyce} the obstruction sheaf in degree $j^\p$ for $(Y, p_X^*T^*_{X, -})$ is $p_X^*\Ac_{T^*_{X, -}}^{(j^\p)}\oplus p_B^*T^*_B$ if $j^\p$ is even and $p_X^*\Ac_{T^*_{X, -}}^{(j^\p)}$ is $j^\p$ is odd. Either way, if $B$ is Stein, then by Lemma \ref{fnkbvhjbvjhdkndkln} and Proposition \ref{rh973hf983hf80jf093fj} we may identify the obstruction \emph{space} of $(Y, p_X^*T^*_{X, -})$ in degree $j^\p$ with $H^1\big(Y, p_X^*\Ac^{(j^\p)}_{T^*_{X, -}}\big)$. For $(Y, T^*_{Y, -})$, its obstruction space in degree $j^\p$ is $H^1\big(Y, \Ac^{(j^\p)}_{T^*_{Y, -}}\big)$. 

\subsubsection{The Filtration on Exterior Powers}
We digress now onto the following observation whose relevance will soon become apparent. With $(Y, T^*_{Y, -})$ the $gt$-model for $(\widetilde\Uc\ra\widetilde \Xc)$, recall that
 $T^*_{Y, -}$ is an extension of $T^*_{X, -}$ by $T^*_{B, -}$, i.e., it fits into an exact sequence $0\ra T^*_{B, -} \ra T^*_{Y, -}\ra T^*_{X, -}\ra 0$. As such, the $j^\p$-th exterior power $\wedge^{j^\p}T^*_{Y, -}$ carries a filtration by $\Oc_Y$-modules $F^{j^\p} = (F_k^{~j^\p})_{k = 0,\ldots, j^\p}$ where:
\begin{enumerate}[$\bt$]
	\item $0 \subset F_{j^\p}^{~j^\p} \subset F_{j^\p-1}^{~j^\p} \subset \cdots \subset F_{0}^{~j^\p} = \wedge^{j^\p}T^*_{Y, -}$
	\item for any $k = 0, \ldots, j^\p$ that:
	\begin{align}
	F^{~j^\p}_k/F^{~j^\p}_{k+1} \cong \wedge^{k}T^*_{B, -}\boxtimes \wedge^{j^\p-k}T^*_{X, -}.
	\label{fjbbvrbviubvouevoineiov}
	\end{align}
\end{enumerate}
With $F^{~j^\p}_{j^\p+1} = (0)$ we evidently have, $F^{~j^\p}_{j^\p} =p_B^* \wedge^{j^\p}T^*_{B, -}$ and $F^{~j^\p}_0/F^{~j^\p}_1 = p_X^*\wedge^{j^\p}T^*_{X, -}$. We are interested in this latter relation. Observe that it gives an exact sequence of $\Oc_Y$-modules $0 \ra F_1^{~j^\p} \ra \wedge^{j^\p}T^*_{Y, -}\ra \wedge^{j^\p}T^*_{X, -}\ra0$. Assuming $T^*_{Y, \pm}$ is locally free, the sheaf-hom functor $\mathcal Hom_{\Oc_Y} (T^*_{Y, (\pm)^{j^\p}}, -)$ will be exact and therefore gives an exact sequence
\begin{align}
0 
\lra 
\mathcal Hom_{\Oc_Y} (T^*_{Y, (\pm)^{j^\p}}, F_1^{~j^\p})
\lra
\Ac_{T^*_{Y, -}}^{(j^\p)}
\lra
\mathcal Hom_{\Oc_Y} (T^*_{Y, (\pm)^{j^\p}}, \wedge^{j^\p}T^*_{X, -})
\lra 0
\label{rg784gf73gf9h398fh3hf83}
\end{align}
where we have used the characterisation of the obstruction sheaf as the hom-sheaf in \eqref{fbvkbvyibiufioenievpmpeoe}, i.e., that  $\Ac_{T^*_{Y, -}}^{(j^\p)} =  \mathcal Hom_{\Oc_Y} (T^*_{Y, (\pm)^{j^\p}}, \wedge^{j^\p}T^*_{Y, -})$.

\subsubsection{Compatibility}
Recall from \eqref{ffgggruepepedjfhffhfj} that an embedding of models $(X, T^*_{X, -})\subset(Y, T^*_{Y, -})$ presumes a surjection $T^*_{Y, (\pm)^{j^\p}}|_X \ra T^*_{X, (\pm)^{j^\p}}$ for  all $j^\p$. By contravariance of the hom-functor we then have a natural transformation $\iota: \mathcal Hom(T^*_{X, \pm}, -) \Rightarrow \mathcal Hom(T^*_{Y, \pm}|_X, -)$. This allows for a comparison of the obstruction sheaf of the total space and the fiber through the following maps:
\begin{align}
\xymatrix{
& & &\Ac_{T^*_{Y, -}}^{(j^\p)}\ar[d]^{p}
\\
p_X^*\Ac_{T^*_{X, -}}^{(j^\p)} \ar[rrr]_\iota & & & \mathcal Hom_{\Oc_Y} \big(T^*_{Y, (\pm)^{j^\p}}, \wedge^{j^\p}T^*_{X, -}\big)
}
\label{regf79f98h89hf08f03}
\end{align}
On cohomology, recalling that $H^1\big(Y, p_X^*\Ac_{T^*_{X, -}}^{(j^\p)}\big) = H^1\big(Y,\Ac_{p_X^*T^*_{X, -}}^{(j^\p)}\big)$ is the obstruction space for $(Y, p_X^*T^*_{X, -})$ when $B$ is Stein, we therefore have the compatibility diagram for the embedding of models $(Y, p_X^*T^*_{X, -})\subset (Y, T^*_{Y, -})$ stated below in the form of a proposition.

\begin{PROP}\label{fjdkbvjvbebvuibvevbeo}
Let $(Y, T^*_{Y, -})$ be a $gt$-family of models over $(B, T^*_{B, -})$ with fiber $(X, T^*_{X, -})$ and $B$ Stein. Then the compatibility diagram for the embedding of models $(Y, p_X^*T^*_{X, -})\subset (Y, T^*_{Y, -})$ fits into the following:
\[
\xymatrix{
& & \vdots\ar[d]
\\
& & H^1(Y, \mathcal Hom_{\Oc_Y}(T^*_{Y, (\pm)^j}, F_1^{~j^\p})\big)\ar[d]
\\
\ar[d] H^1\big(Y, \Ac_{p_X^*T^*_{X, -}; T^*_{Y, -}}^{(j^\p)}\big)\ar[rr] & &H^1\big(Y, \Ac^{(j^\p)}_{T^*_{Y, -}}\big) \ar[d]^{p_*}
\\
 H^1\big(Y, p_X^*\Ac^{(j^\p)}_{T^*_{X, -}}\big) \ar[rr]^{\iota_*} & 
 & H^1\big(Y, \mathcal Hom_{\Oc_Y} (T^*_{Y, (\pm)^{j^\p}}, \wedge^{j^\p}T^*_{X, -})\big)\ar[d]
 \\
 & & H^2(Y, \mathcal Hom_{\Oc_Y}(T^*_{Y, (\pm)^{j^\p}}, F_1^{~j^\p})\big) \ar[d]
 \\
 & & \vdots
}
\]
where the vertical column on the right-most side is exact and induced from \eqref{rg784gf73gf9h398fh3hf83}. In particular, we have the comparison of obstruction classes,
\begin{align}
\iota_*\om(\Uc\ra\Xc) = p_*\om(\widetilde\Uc\ra\widetilde\Xc)
\label{fjvkbvurbvuevioenvoinev}
\end{align}
where $(\Uc\ra\Xc\ra B)$ is the underlying $gtm$-family of $(\widetilde\Uc\ra\widetilde\Xc\ra\widetilde B)$.
\qed
\end{PROP}

\section{$gtm$-Families over Superspaces: Obstruction Maps}
\label{fjvbgbvirbuiburnvoenoinveo}

\noindent 
In Example \ref{reugf78hf98h30f830} and \ref{rgf784gf73hf83h0j9f33333} it was asserted that, to any deformation of a super Riemann surface $(\Uc\ra\Xc\ra \Cbb^{0|1})$, the class of its model coincides with the obstruction to splitting the total space $(\Uc\ra\Xc)$. This relation is suggestive of a more general role played by the class of a $gt$-model and the families it defines. We propose that it plays a role in obstructing the existence of families.

\subsection{Refinements and Secondary Obstructions}

\subsubsection{Refinements of Splitting Type}
Let $(Y, T^*_{Y,-})$ be a $gt$-model over $(B, T^*_{B, -})$ with fiber $(X, T^*_{X, -})$. With the filtration on $\wedge^{j^\p}T^*_{Y, -}$ by the sheaves $F^{j^\p} = (F^{~j^\p}_k)_{k = 0, \ldots, j^\p}$ leading to \eqref{fjbbvrbviubvouevoineiov}, we have a filtratration of the obstruction sheaf
\begin{align*}
\mathcal Hom_{\Oc_Y}\big(T^*_{Y, (\pm)^{j^\p}},F^{~j^\p}_{j^\p}\big)
\subset  
\mathcal Hom_{\Oc_Y}&\big(T^*_{Y, (\pm)^{j^\p}},F^{~j^\p}_{j^\p-1}\big)
\\
&
\subset 
\cdots
\subset
\mathcal Hom_{\Oc_Y}\big(T^*_{Y, (\pm)^{j^\p}},F^{~j^\p}_{1}\big)
\subset
\mathcal A_{T^*_{Y, -}}^{(j^\p)}
\end{align*}
where the latter inclusion above is the map in \eqref{rg784gf73gf9h398fh3hf83}. The above inclusions of sheaves induce maps on cohomology. We are interested in the map in degree one. For each $k = 1, \ldots, j^\p$ we have:
\[
\iota_{k*}:
H^1\big(Y,\mathcal Hom_{\Oc_Y}\big(T^*_{Y, (\pm)^{j^\p}},F^{~j^\p}_{k}\big)\big)
\lra 
H^1\big(Y, \mathcal A_{T^*_{Y, -}}^{(j^\p)}\big).
\]
Given a family of supermanifolds $(\widetilde\Uc\ra\widetilde\Xc\ra \widetilde B)$ of splitting type $j^\p$ and obstruction class $\om(\widetilde\Uc\ra\widetilde\Xc)$, there is no guarantee that $\om(\widetilde\Uc\ra\widetilde\Xc)$ will lie in the image of $\iota_{k*}$ for some $k$. Indeed, if $\om(\Uc\ra\Xc)$ is the obstruction to splitting the underlying family and $\iota_*\om(\Uc\ra\Xc) \neq0$, then \eqref{fjvkbvurbvuevioenvoinev} implies $\om(\widetilde\Uc\ra\widetilde\Xc)$ cannot lie in the image of $\iota_{k*}$ for \emph{any} $k$. We will be primarily interested in the case where $\iota_*\om(\Uc\ra\Xc) =0$.

\begin{DEF}
\emph{A $gtm$-family of supermanifolds over a Stein, superspace base is said to have a \emph{refined splitting type} if the obstruction to splitting its underlying family vanishes in the correspondence space, i.e., if its image under $\iota_*$ in the diagram in Proposition \ref{fjdkbvjvbebvuibvevbeo} is trivial.}
\end{DEF}

\noindent
By exactness of the right-most vertical column in the diagram in Proposition \ref{fjdkbvjvbebvuibvevbeo}, that a $gtm$-family of supermanifolds will have a refined splitting type is equivalent to the statement: \emph{its obstruction to splitting lies in the image of $\iota_{k*}$ for some $k$}. 

\begin{DEF}\label{rfbygf874gf93hf0j39fj3}
\emph{A $gtm$-family of supermanifolds over a Stein, superspace base is said to have \emph{refined splitting type $j^\p = a + b$} for some $b> 0$ if
\begin{enumerate}[(i)]
	\item it has splitting type $j^\p$ and;
	\item its obstruction splitting lies in the image of $\iota_{b*}$ and \emph{not} $\iota_{b+1*}$.
\end{enumerate}}
\end{DEF}

\subsubsection{The Secondary Obstruction}
For each $a, b$ with $a+b = j^\p$ we have the quotient from \eqref{fjbbvrbviubvouevoineiov}, which is here $F^{~j^\p}_b/ F^{~j^\p}_{b+1} \cong \wedge^{b}T^*_{B, -}\boxtimes \wedge^{a}T^*_{X, -}$. For convenience we set:
\[
{\bf H}^{a, b}_n(Y, T^*_{Y, -}) 
\stackrel{\Delta}{=}
H^n\big(Y, \mathcal Hom_{\Oc_Y}(T^*_{Y, (\pm)^{a+b}}, \wedge^{b}T^*_{B, -}\boxtimes \wedge^{a}T^*_{X, -})\big).
\]
With the short exact sequence $F^{~j^\p}_{b+1}\ra F^{~j^\p}_{b} \ra F^{~j^\p}_b/ F^{~j^\p}_{b+1}$ we obtain the following diagram of cohomology spaces 
\begin{align}
\xymatrix{
\ar[d]^{\iota_{b+1*}}H^1\big(Y, \mathcal Hom_{\Oc_Y}(T^*_{Y, (\pm)^{j^\p}}, F^{~j^\p}_{b+1} )\big)\ar[r]
&\ar[d]^{\iota_{b*}} H^1\big(Y, \mathcal Hom_{\Oc_Y}(T^*_{Y, (\pm)^{j^\p}}, F^{~j^\p}_{b} )\big)\ar[r]& {\bf H}^{a, b}_1(Y, T^*_{Y, -})
\\
H^1\big(Y, \Ac_{T^*_{Y, -}}^{(j^\p)}\big)\ar@{=}[r]& H^1\big(Y, \Ac_{T^*_{Y, -}}^{(j^\p)}\big)&
}
\label{rf478gf74f98hf830f9j3}
\end{align}
If a $gtm$-family has refined splitting type $j^\p = a+b$, then it obstruction to splitting will lie in the image of $\iota_{b*}$. Its projection onto a class in ${\bf H}^{a, b}_1(Y, T^*_{Y, -})$ will be referred to as the \emph{secondary obstruction} to splitting the family. As the top row of spaces in \eqref{rf478gf74f98hf830f9j3} is exact, Definition \ref{rfbygf874gf93hf0j39fj3}(ii) amounts to the statement: \emph{the secondary obstruction to splitting a $gtm$-family with refined splitting type is non-vanishing}.

\subsection{The Obstruction Maps}
In taking motivation from an analogous problem for integrating thickenings to supermanifolds in \cite{BETTPHD, BETTOBSTHICK}, we consider the following question:

\begin{QUE}\label{rfh849hf84jf0jf933333}
Fix a $gt$-model $(Y, T^*_{Y, -})$. When will a class in ${\bf H}^{a, b}_1(Y, T^*_{Y, -})$ be a secondary obstruction to splitting some $gtm$-family of supermanifolds with total space model $(Y, T^*_{Y, -})$?
\end{QUE}

\noindent
Following on from \eqref{rf478gf74f98hf830f9j3}, observe that we have a composition for each $p$ defining the map $\pt_p^{\{a, b\}}$:
\begin{align}
\xymatrix{
\ar@{-->}[drr]|{\pt_p^{\{a, b\}}} {\bf H}^{a, b}_p(Y, T^*_{Y, -}) \ar[rr] & & H^{p+1}\big(Y, \mathcal Hom_{\Oc_Y}(T^*_{Y, (\pm)^{j^\p}}, F^{~j^\p}_{b+1} )\big)\ar[d]
\\
& & {\bf H}^{a-1, b+1}_{p+1}(Y, T^*_{Y, -})
}
\label{rf748f4hf93fjf030fj390}
\end{align}
The obstruction spaces of a model give rise to a differential complex, referred to in \cite{BETTAQ} as the obstruction complex. Similarly, we will see that \eqref{rf748f4hf93fjf030fj390} will likewise give a differential complex.

\begin{PROP}\label{rg78gf87gf973h98hf08h3}
The collection $\big(\pt^{\{a, b\}}_p : {\bf H}^{a, b}_p(Y, T^*_{Y, -})\ra{\bf H}^{a-1, b+1}_{p+1}(Y, T^*_{Y, -})\big)_{p\in \Zbb}$ defines a bounded, differential complex for any $gt$-model $(Y, T^*_{Y, -})$.
\end{PROP}

\begin{proof}
For either dimensional reasons or general properties of exterior powers, the cohomology group ${\bf H}^{a, b}_{p}(Y, T^*_{Y, -})$ will be trivial for $a + b > \min\{\dim Y, \mathrm{rank}~T^*_{Y, -}\}$. Hence, to confirm the proposition, it remains to show $\pt^{\{a-1, b+1\}}_{p+1}\circ \pt^{\{a, b\}}_p = 0$. Since these maps are, respectively, composites we have the following commutative diagram:
\[
\xymatrix{
{\bf H}^{a, b}_p(Y,T^*_{Y, -}) \ar@{-->}[dr]_{\pt^{\{a, b\}}_p}  \ar[r] & H^{p+1}\big(Y, \mathcal Hom_{\Oc_Y}(T^*_{X, \pm}, F^{~j^\p}_{b+1}\big)\big) \ar[d]
\\
& {\bf H}^{a-1, b+1}_{p+1}(Y,T^*_{Y, -})\ar[d]\ar@{-->}[dr]^{\pt^{\{a-1, b+1\}}_{p+1}}
\\
&
H^{p+2}\big(Y, \mathcal Hom_{\Oc_Y}(T^*_{X, \pm}, F^{~j^\p}_{b+2}\big)\big) \ar[r]
 & {\bf H}^{a-2, b+2}_{p+2}(Y,T^*_{Y, -})
}
\]
where $a+b=j^\p$. The composition $\pt^{\{a-1, b+1\}}_{p+1}\circ \pt^{\{a, b\}}_p$ will vanish since (i) the above diagram commutes and; (ii) the vertical column in the above diagram is exact. 
\end{proof}

\noindent
A partial resolution to Question \ref{rfh849hf84jf0jf933333} is now immediate. 

\begin{PROP}\label{rfh74fg874f93hf83j0f93}
If a class $\nu\in {\bf H}^{a, b}_1(Y,T^*_{Y, -})$ can be identified with a secondary obstruction to splitting a $gtm$-family with total space model $(Y, T^*_{Y, -})$, then it is necessary for $\pt^{\{a, b\}}_1(\nu) = 0$.\qed
\end{PROP}

\noindent
As a consequence of Proposition \ref{rg78gf87gf973h98hf08h3} and \ref{rfh74fg874f93hf83j0f93} we propose the following definition.

\begin{DEF}\label{fjvkbvrbvbuvbenvoeineoddd}
\emph{For any $gt$-model $(Y, T^*_{Y, -})$, the complex in Proposition \ref{rg78gf87gf973h98hf08h3} will be referred to as its \emph{secondary obstruction complex}. The differentials in degree one, i.e., the maps $\pt^{\{a, b\}}_1$ in Proposition \ref{rfh74fg874f93hf83j0f93}, will be referred to as \emph{obstruction maps}.}
\end{DEF}

\noindent
In what follows we will construct another map ${\bf H}^{a, b}_p(Y,T^*_{Y, -})\ra {\bf H}^{a-1, b+1}_{p+1}(Y,T^*_{Y, -})$ by reference to the class of the model $(Y, T^*_{Y, -})$. We will conjecture that the image of this map will be contained in the image of the differentials in the secondary obstruction complex, thereby explicitly revealing a role played by the model class in obstructing the existence of families with refined splitting type. We present an affirmation of this conjecture in the case $a = 1$.

\section{Obstruction Maps and the Model Class}

\noindent
We firstly note the following immediate relation. Let $(Y, T^*_{Y, -})$ be a $gt$-model. Since its model class $\Theta(T^*_{Y, -})$ represents the obstruction to $T^*_{Y, -}$ being split we can deduce:

\begin{PROP}
Let $(Y, T^*_{Y, -})$ be a $gt$-model with trivial model class. Then the differentials in its secondary complex are all trivial.
\end{PROP}

\begin{proof}
By construction in \eqref{rf748f4hf93fjf030fj390} the differential $\pt^{\{a, b\}}_p$ has as a composite the boundary map $ {\bf H}^{a, b}_p(Y, T^*_{Y, -})\ra H^{p+1}\big(Y, \mathcal Hom_{\Oc_Y}(T^*_{Y, (\pm)^{j^\p}}, F^{~j^\p}_{b+1} )\big)$. This map will vanish if the model class vanishes and, therefore, so will $\pt^{\{a, b\}}_p$.
\end{proof}

\subsection{The Model Class Map}
Presently, we aim to get a better understanding of the secondary obstruction complex in the case where the model class is not necessarily trivial. This will entail deriving explicit and general relations between the model class and the boundary maps in the secondary obstruction complex. To begin, let $\Ac$ be an abelian sheaf on a space $Y$. Then the $\Ac$-valued $p$-cochains $C^p(Y, \Ac)$ will be an abelian group.

\begin{REM}\label{rfg784gf97hf8wwwwhf0}
\emph{The $p$-cochains on $Y$, $C^p(Y, -)$, defines a \emph{covariant} functor from the category of $\Oc_Y$-modules to abelian groups. This means, if $\Fc\ra \Gc$ is a morphism of $\Oc_Y$-modules, we have a corresponding morphism of abelian groups $C^p(Y, \Fc)\ra C^p(Y, \Gc)$.}
\end{REM}

\noindent
Now for two abelian sheaves $\Ac$ and $\Bc$, their tensor product induces a map on cochains $C^q(Y, \Ac) \otimes C^p(Y, \Bc) \ra C^{q+p}(Y, \Ac\otimes \Bc)$ given by,
\begin{align}
(\phi\otimes \vp)_{i_0\cdots i_{p+q}}
=
\phi_{i_0\cdots i_p}\otimes \vp_{i_{p}\cdots i_{p+q}}
\label{fvbbvuibeviubeuvueoe}
\end{align}
on a $(p+q)$-fold intersection $U_{i_0}\cap\cdots\cap U_{i_{p+q}}$. As remarked in \cite[pp. 29-30]{BRY}, the map induced by \eqref{fvbbvuibeviubeuvueoe} on \v Cech cohomology defines the cup product. It coincides with the cup product on sheaf cohomology when $Y$ has sufficiently nice topological properties, i.e., is paracompact and Hausdorff. For $(Y, T^*_{Y, -})$ a $gt$-model over $(B, T^*_{B, -})$ with fiber $(X, T^*_{X, -})$, consider the sheaves:
\begin{align*}
\Ac = \mathcal Hom_{\Oc_Y}\big(T^*_{X, -}, T^*_{B, -}\big)
&&
\mbox{and}
&&
\Bc^{j^\p} = \mathcal Hom_{\Oc_Y}\big(T^*_{Y, (\pm)^{j^\p}}, \wedge^bT^*_{B, -}\boxtimes \wedge^aT^*_{X, -}\big).
\end{align*}
For each $a$ we have a local map $\wedge^aT^*_{X, -} \ra T^*_{X, -}\otimes \wedge^{a-1}T^*_{X, -}$ given by sending an $a$-fold wedge product $\q_{i_1}\wedge\cdots\wedge\q_{i_a}$ to its image under the derivation $\frac{1}{a!}\sum_\al \q_\al\pt/\pt \q_\al$. 
This need not necessarily give a morphism of sheaves, but we can nevertheless get a morphism of $p$-cochains $C^p(Y, \wedge^aT^*_{X, -}) \ra C^p(Y, T^*_{X, -}\otimes \wedge^{a-1}T^*_{X, -})$, $p\geq0$, inducing the map:
\begin{align}
C^p\big(Y, \Bc^{j^\p}\big) 
\lra 
C^p\big(Y,  \mathcal Hom_{\Oc_Y}\big(T^*_{Y, (\pm)^{j^\p}}, \wedge^bT^*_{B, -}\boxtimes T^*_{X, -}\boxtimes \wedge^{a-1}T^*_{X, -}\big)\big).
\label{rhf74gf74hf8f0j390fj930}
\end{align}
Now very generally we have a product of hom-spaces in some fixed category of modules:
\begin{align*}
\Hom(A, B)\otimes \Hom(B, C) \lra \Hom(A, C) 
&&
\mbox{given by}
&&
f\otimes g \longmapsto \langle f, g\rangle = g\circ f.
\end{align*}
This generalises to
\begin{align}
\langle-,-\rangle
:
\Hom(A, B\otimes F)\otimes \Hom(B, C) \lra \Hom(A,  C\otimes F) 
\label{rjfbigf73gf73hf3fj3fpoj}
\end{align}
For any module $F$. With the map \eqref{rhf74gf74hf8f0j390fj930} and the product in \eqref{rjfbigf73gf73hf3fj3fpoj} above we can form the following composition of cochains:
\begin{align*}
(-, -)^{\{a, b\}}: C^q\big(Y, \mathcal H&om_{\Oc_Y}\big(T^*_{X, -}, T^*_{B, -}\big)\big)
\\
&\otimes 
C^p\big(Y, \mathcal Hom_{\Oc_Y}\big(T^*_{Y, (\pm)^{j^\p}}, \wedge^bT^*_{B, -}\boxtimes \wedge^aT^*_{X, -}\big)\big)
\\
&
\lra
C^{q+p}\big(Y, \mathcal Hom_{\Oc_Y}\big(T^*_{Y, (\pm)^{j^\p}}, \wedge^bT^*_{B, -}\boxtimes T^*_{B, -}\boxtimes \wedge^{a-1}T^*_{X, -}\big)\big)
\\
&\lra 
C^{q+p}\big(Y, \mathcal Hom_{\Oc_Y}\big(T^*_{Y, (\pm)^{j^\p}}, \wedge^{b+1}T^*_{B, -}\boxtimes \wedge^{a-1}T^*_{X, -}\big)\big)
\end{align*}
where the latter map above is induced from the wedge product $\wedge^b T^*_{B, -}\otimes T^*_{B, -} \ra \wedge^{b+1}T^*_{B, -}$ (c.f., Remark \ref{rfg784gf97hf8wwwwhf0}). The composition $(-,-)^{\{a, b\}}$ above gives the following map on cohomology:
\[
(-,-)_*^{\{a, b\}}
:
H^q\big(Y, \mathcal Hom_{\Oc_Y}\big(T^*_{X, -}, T^*_{B, -}\big)\big)
\otimes 
{\bf H}^{a, b}_p(Y, T^*_{Y, -})
\lra
{\bf H}^{a-1, b+1}_{p+q}(Y, T^*_{Y, -})
\]
On identifying $H^q\big(Y, \mathcal Hom_{\Oc_Y}\big(T^*_{X, -}, T^*_{B, -}\big)\big)$ with $\mathrm{Ext}_{\Oc_Y}^q(T^*_{X, -}, T^*_{B, -})$ note that the model class $\Theta(T^*_{Y, -})$ defines an element in $\mathrm{Ext}_{\Oc_Y}^1(T^*_{X, -}, T^*_{B, -})$. Hence for $q = 1$ we have the following map:\footnote{We have omitted reference to the subscript $p$, labelling the cohomological degree, so as to avoid cumbersome notation.}
\begin{align}
\big(\Theta(T^*_{Y, -}), -\big)_*^{\{a, b\}}
:
{\bf H}^{a, b}_p(Y, T^*_{Y, -})
\lra
{\bf H}^{a-1, b+1}_{p+1}(Y, T^*_{Y, -}).
\label{rfg684g874gf9h8f9h38f3}
\end{align}
The map in \eqref{rfg684g874gf9h8f9h38f3} will be referred to as the \emph{model class map}.

\subsection{A Conjectural Relation} 
Our conjecture below concerns a comparison of the model class map with the differentials $\pt^{\{a, b\}}$ in the secondary obstruction complex of a $gt$-model.

\begin{CONJ}\label{rhf74gf74gf983hf80309fj30}
For any $gt$-model $(Y, T^*_{Y, -})$ and any $a, b$, we have a containment of images
\[
\img~\big(\Theta(T^*_{Y, -}), -\big)_*^{\{a, b\}}
\subseteq 
\img~\pt^{\{a, b\}}.
\]
\end{CONJ}

\noindent
We establish Conjecture \ref{rhf74gf74gf983hf80309fj30} in the following special case.

\begin{PROP}\label{fjbbvubiuvbuenoienienp}
Conjecture \ref{rhf74gf74gf983hf80309fj30} is true in the case $a = 1$.
\end{PROP}

\noindent
We defer the proof of Proposition \ref{fjbbvubiuvbuenoienienp} to Appendix \ref{rfh97hfg9hf83jf09jf9j333}. A perhaps worthwhile observation to make here is that, while the motivation behind Conjecture \ref{rhf74gf74gf983hf80309fj30} involves supermanifolds, the statement is itself independent of supermanifolds. It only requires understanding what is meant by a `$gt$-model', as given in Definition \ref{rfh74gf9hf830fj93jf}, and is therefore potentially interesting in its own right. Proposition \ref{rfh74fg874f93hf83j0f93} then allows us to return to the setting of supermanifolds, if desired. This is illustrated in the following consequence.

\begin{COR}\label{rfg64gf84gf794hf98h}
Let $(Y, T^*_{Y, -})$ be a $gt$-model with non-trivial model class $\Theta(T^*_{Y, -})$. Then, for any $j^\p$, let $\nu\in {\bf H}^{1, j^\p-1}_1(Y, T^*_{Y, -})$ be such that $\big(\Theta(T^*_{Y, -}), \nu\big)_*^{\{1, j^\p-1\}}\neq0$. Then there exists a class $\nu^\p \in  {\bf H}^{1, j^\p-1}_1(Y, T^*_{Y, -})$ which cannot be realised as a secondary obstruction to splitting any $gtm$-family of supermanifolds.
\end{COR}

\begin{proof}
Generally, given vector spaces $V, W$ and maps $f, g: V\ra W$ such that $\img~f\subseteq \img~g$, this means, for any $a\in A$, that $f(a) = g(a^\p)$ for some $a^\p$. Hence if $f(a)\neq0$ then $g(a^\p)\neq0$. The present statement now follows by considering the case where $f = \big(\Theta(T^*_{Y, -}), -\big)_*^{\{1, j^\p-1\}}$, $g = \pt^{\{1, j^\p-1\}}$ and applying Proposition \ref{rfh74fg874f93hf83j0f93} and \ref{fjbbvubiuvbuenoienienp}.
\end{proof}

\begin{REM}
\emph{From the proof of Proposition \ref{fjbbvubiuvbuenoienienp} the classes $\nu$ and $\nu^\p$ in Corollary \ref{rfg64gf84gf794hf98h} are related through the cup product. That is, in the notation of the proof, we have the relation $\nu^\p = \tau_*\big(\Theta(T^*_{Y, -})\smile \nu\big)$. Hence we could further deduce, if $\nu$ is such that $\Theta(T^*_{Y, -})\smile \nu = 0$, then $\pt^{\{1, j^\p-1\}}\nu = 0$ leading subsequently to a containment of kernels, complementarily to Proposition \ref{fjbbvubiuvbuenoienienp}.}
\end{REM}

\section{Obstruction Complexes: Questions and Conjectures}
\label{rygf8g478gf794hf893h893h03333333}

\noindent
The deliberations so far in this part of the article, in conjunction with previous work by the author, prompts some natural, general questions which we present here.

\subsection{The Obstruction Complex of a Model}
 Firstly, to recall, in \cite{BETTAQ} a complex, called the `obstruction complex' was derived for any model $(X, T^*_{X, -})$. In the language of \cite{BETTOBSTHICK}, this complex concerned the problem of integrating thickenings to supermanifolds. Denote the obstruction complex of a model $(X, T^*_{X, -})$ by $\mathrm{Ob}(X, T^*_{X, -})$. It is a bi-graded complex with a single differential. The $(p, q)$-th term in $\mathrm{Ob}(X, T^*_{X, -})$ is $H^p\big(X, \Ac_{T^*_{X, -}}^{(q)}\big)$, where $\Ac_{T^*_{X, -}}^{(q)}$ is the $q$-th obstruction sheaf; and the differential is a map  $\dt^{p, q}: \mathrm{Ob}^{p, q}(X, T^*_{X, -}) \ra \mathrm{Ob}^{p+1, q+1}(X, T^*_{X, -})$ and therefore suggests this complex can be derived from some larger complex. This prompts:
 
\begin{QUE}
To any model $(X, T^*_{X. -})$, does there exist some double complex whose total complex is the obstruction complex of $(X, T^*_{X. -})$?
\end{QUE}

\noindent
In the first part of this article we looked at $gt$-models over a connected, Stein base $B$. Let $(Y, T^*_{Y, -})$ be such a model with fiber $(X, T^*_{X, -})$. Recall that $Y = X\times B$ and $T^*_{Y,-} = p_X^*T^*_{X, -}$. Since the obstruction class to splitting the fiber of any $gtm$-family with total space model $(Y, T^*_{Y,-})$ is essentially constant by Theorem \ref{fbjhvvbennciuebcuecoe}, we propose:

\begin{CONJ}\label{kfnvjkfbvhjbvbeuboe}
The obstruction complexes $\mathrm{Ob}(X, T^*_{X, -})$ and $\mathrm{Ob}(Y,p_X^* T^*_{X, -})$ are chain homotopic.
\end{CONJ}

\subsection{The Obstruction Complex of $gt$-Models}
We turn now to $gt$-models over superspace bases. Let $(Y, T^*_{Y, -})$ be a $gt$-model over $(B, T^*_{B, -})$ with fiber $(X, T^*_{X, -})$. We assume $B$ is connected and Stein. Then we have the following complexes associated to $(Y, T^*_{Y, -})$, all related to an obstruction problem:
\begin{enumerate}[$\bt$]
	\item the obstruction complex of the total space model $\mathrm{Ob}(Y, T^*_{Y, -})$;
	\item the obstruction complex of the underlying model $\mathrm{Ob}(Y,p_X^* T^*_{X, -})$;
	\item the obstruction complex of the fiber $\mathrm{Ob}(X, T^*_{X, -})$ and;
	\item the secondary obstruction complex, denoted $\mathrm{Ob}^\p(Y, T^*_{Y, -})$.
\end{enumerate}
It is natural to then wonder if there are any interesting relations between the above complexes, seeing as they can all be `derived' from a single model. The compatibility diagram in Proposition \ref{fjdkbvjvbebvuibvevbeo}, prompts the following question:

\begin{QUE}\label{6g84fg799fh308hf039}
Does there exist a compatibility diagram between the complexes associated to a $gt$-model $(Y, T^*_{Y, -})$, in analogy with that in Proposition \ref{fjdkbvjvbebvuibvevbeo}?
\end{QUE}

\noindent
A potential resolution of Question \ref{6g84fg799fh308hf039} above might be of the following kind. If we take Conjecture \ref{kfnvjkfbvhjbvbeuboe} as given, we may identify $\mathrm{Ob}(Y,p_X^* T^*_{X, -})$ and $\mathrm{Ob}(X, T^*_{X, -})$ up to homotopy. Then, to generalise the exact column in Proposition \ref{fjdkbvjvbebvuibvevbeo}, a compatibility between obstruction complexes could mean: \emph{there exists an exact sequence of complexes $\mathrm{Ob}^\p(Y, T^*_{Y, -})\ra \mathrm{Ob}(Y, T^*_{Y, -})\ra \mathrm{Ob}(X, T^*_{X, -})$, up to homotopy}.

\appendix
\numberwithin{equation}{section}

\section{Proof of Proposition $\ref{fjbbvubiuvbuenoienienp}$}
\label{rfh97hfg9hf83jf09jf9j333}

\noindent
In the case where $a= 1$ the relevant exact sequence coming from the filtration $F^{j^\p}$ on $T^*_{Y, -}$ is $0\ra F^{~j^\p}_{j^\p} \ra F^{~j^\p}_{j^\p-1} \ra \wedge^{j^\p-1}T^*_{B, -}\boxtimes T^*_{X, -}\ra0$. Recall from \eqref{fjbbvrbviubvouevoineiov} that $F^{~j^\p}_{j^\p} = \wedge^{j^\p}T^*_{B, -}$. We want to compare $\pt^{\{1, j^\p-1\}}$ and $(\Theta(T^*_{Y, -}), -)^{\{1, j^\p-1\}}_*$ which are maps:
\begin{align*}
{\bf H}^{1, j^\p-1}_1(Y, T^*_{Y, -}) = H^1\big(Y, &\mathcal Hom_{\Oc_Y}(T^*_{Y, (\pm)^{j^\p}}, T^*_{X, -}\boxtimes \wedge^{j^\p-1}T^*_{B, -})\big)
\\
&\lra
H^2\big(Y, \mathcal Hom_{\Oc_Y}(T^*_{Y, (\pm)^{j^\p}}, \wedge^{j^\p}T^*_{B, -})\big)
=
{\bf H}^{0, j^\p}_2(Y, T^*_{Y, -})
\end{align*}
Inspection of the map $(\Theta(T^*_{Y, -}), -)^{\{1, j^\p-1\}}_*$ reveals that, by its construction, it will coincide with
the image of the cup-product-with-$\Theta(T^*_{Y, -})$ in ${\bf H}^{0, j^\p}_2(Y, T^*_{Y, -})$. That is,
\begin{align}
(\Theta(T^*_{Y, -}),\nu)^{\{1, j^\p-1\}}_*
= \tau_*\big(\Theta(T^*_{Y, -})\smile \nu\big).
\label{fjbjvyueviubonoinioeipe}
\end{align}
for any $\nu$. The map $\tau_*$ will be described after the following relevant digression on the characterisation of cup products from \cite{BREDSHEAF}. If $0\ra \Ac^\p\ra \Ac\ra \Ac^{\p\p}\ra0$ is a short exact sequence of sheaves of modules on a space $Y$ and $\Bc$ is any flat module on $Y$, i.e., that $(-)\otimes \Bc$ defines an exact functor, then the following diagram commutes:
\begin{align}
\xymatrix{
\ar[dd] H^p(Y, \Ac^{\p\p})\otimes H^q(Y, \Bc) \ar[r] & H^{p+q}(Y, \Ac^{\p\p}\otimes \Bc)\ar[d]
\\
& H^{p+q+1}(Y, \Ac^\p\otimes \Bc) \ar@{=}[d]
\\
H^{p+1}(Y, \Ac^\p)\otimes H^q(Y, \Bc) \ar[r]  &  H^{p+q+1}(Y, \Ac^\p\otimes \Bc)
}
&&
\xymatrix{
u\otimes v \ar@{|->}[dd]\ar@{|->}[r] & u\smile v\ar@{|->}[d]
\\
& \dt^\p(u\smile v)\ar@{=}[d]
\\
\dt(u)\otimes v\ar@{|->}[r] &\dt(u)\smile v
}
\label{rfh49f984hf8hf0833}
\end{align}
where $\smile$ denotes the cup product and $\dt, \dt^\p$ are the following respective boundary maps on cohomology, $\dt : H^p(Y, \Ac^{\p\p})\ra H^{p+1}(Y, \Ac^\p)$ and $\dt^\p : H^{p+q}(Y, \Ac^{\p\p}\otimes \Bc)\ra H^{p+q+1}(Y, \Ac^\p\otimes \Bc)$. This is the content of \cite[Theorem 7.1(b), p. 57]{BREDSHEAF}. We want to specialise now to the following case, which we firstly present abstractly. Consider a short exact sequence of sheaves of flat modules $0\ra \Mcl^\p\ra\Mcl\ra\Mcl^{\p\p} \ra0$ on $Y$. Then $\mathcal Hom_{\Oc_Y}(\Mcl^{\p\p}, -)$ will be exact. Now set $\Ac^\p = \mathcal Hom_{\Oc_Y}(\Mcl^{\p\p}, \Mcl^\p)$, $\Ac^{\p\p} = \mathcal Hom_{\Oc_Y}(\mathcal M^{\p\p}, \Mcl^{\p\p})$. With some flat module $\Bc$ we obtain the following diagram for $p = 0$ and any $q$:
\begin{align}
\xymatrix{
H^0\big(Y, \mathcal Hom_{\Oc_Y}(\Mcl^{\p\p}, \Mcl^{\p\p})\big) \otimes H^q\big(Y, \Bc\big)
\ar[d]_{\dt\otimes 1}\ar[r] 
& 
H^q\big(Y,\mathcal Hom_{\Oc_Y}(\Mcl^{\p\p}, \Mcl^{\p\p})\otimes \Bc\big) \ar[d]^{\dt^\p}
\\
H^1\big(Y,\mathcal Hom_{\Oc_Y}(\Mcl^{\p\p},\Mcl^\p)\big)
\otimes 
H^q\big(Y, \Bc\big)
 \ar[r] & 
 H^{q+1}(Y, \mathcal Hom_{\Oc_Y}(\Mcl^{\p\p},\Mcl^\p)\otimes \Bc\big)
}
\end{align}
The extension class of $\Mcl$ coincides with $\dt({\bf 1}_{\Mcl^{\p\p}})\in H^1\big(Y,\mathcal Hom_{\Oc_Y}(\Mcl^{\p\p},\Mcl^\p)\big)$. Denote the extension class by $\Theta(\Mcl)$. The cup product formula in \eqref{rfh49f984hf8hf0833} gives:
\begin{align}
\Theta(\Mcl)\smile \nu = \dt^\p({\bf1}_{\Mcl^{\p\p}}\smile \nu)
\label{77757h98hf8h3fh830}
\end{align}
for any $x\in H^q(Y,\Bc)$. Now consider another short exact sequence $0\ra \Nc^\p\ra \Nc\ra \Nc^{\p\p}\ra0$. Suppose $\Bc$ is such that $\mathcal Hom_{\Oc_Y}(\Mcl^{\p\p}, -)\otimes \Bc$ defines a morphism from the extension defining $\Mcl$ to that defining $\Nc$, i.e., that we have a morphism of exact sequences:
\begin{align}
\xymatrix{
 \ar[d] \mathcal Hom_{\Oc_Y}(\Mcl^{\p\p}, \Mcl^\p)\otimes \Bc \ar[r] & \ar[d]\mathcal Hom_{\Oc_Y}(\Mcl^{\p\p}, \Mcl) \otimes \Bc\ar[r] &\ar[d] \mathcal Hom_{\Oc_Y}(\Mcl^{\p\p}, \Mcl^{\p\p})\otimes \Bc
\\
\Nc^\p\ar[r] & \Nc\ar[r]  & \Nc^{\p\p}
}
\label{rfh94hf89hf80f093}
\end{align}
Then by naturality of cohomology we can continue the diagram in \eqref{rfh49f984hf8hf0833} one step to the right to get the following commutative diagram:
\begin{align}
\xymatrix{
H^q\big(Y,\mathcal Hom_{\Oc_Y}(\Mcl^{\p\p}, \Mcl^{\p\p})\otimes \Bc\big) \ar[d]_{\dt^\p}\ar[rr] && H^q\big(Y, \Nc^{\p\p}\big)\ar[d]^{\pt}
\\
H^{q+1}(Y, \mathcal Hom_{\Oc_Y}(\Mcl^{\p\p},\Mcl^\p)\otimes \Bc\big)\ar[rr] 
&&
H^{q+1}\big(Y, \Nc^\p\big)
}
\label{kdjghghrrejekkdldkkff}
\end{align}
Denote by $\tau_*$ the horizontal maps in \eqref{kdjghghrrejekkdldkkff} above. Then from the identity in \eqref{77757h98hf8h3fh830} along with commutativity of \eqref{kdjghghrrejekkdldkkff} we conclude:
\begin{align}
\tau_*\big(\Theta(\Mcl)\smile \nu\big)
&=
\tau_*\dt^\p\big({\bf1}_{\Mcl^{\p\p}}\smile \nu\big)
\notag
\\
&=
\pt \tau_*\big({\bf1}_{\Mcl^{\p\p}}\smile \nu\big).
\label{fjbjfbvybuveoioeeee}
\end{align}
Thus we see above how the cup product can be related to the boundary map $\pt$ of an arbitrary short exact sequence equipped and a morphism in \eqref{rfh94hf89hf80f093}. To complete the proof of this proposition now, set:
\begin{enumerate}[$\bt$]
	\item $\Mcl^\p = T^*_{B, -}$; 
	\item $\Mcl = T^*_{Y, -}$ and;
	\item $\Mcl^{\p\p} = T^*_{X, -}$;
	\item $\Nc^\p = \mathcal Hom_{\Oc_Y}\big(T^*_{X, \pm}, F^{~j^\p}_{j^\p}\big)$;
	\item $\Nc = \mathcal Hom_{\Oc_Y}\big(T^*_{X, \pm}, F^{~j^\p}_{j^\p-1}\big)$;
	\item $\Nc^{\p\p} =  \mathcal Hom_{\Oc_Y}\big(T^*_{X, \pm}, T^*_{X, -}\boxtimes \wedge^{j^\p-1}T^*_{B, -}\big)$
\end{enumerate}
and take $\Bc = \Nc^{\p\p} = \mathcal Hom_{\Oc_Y}\big(T^*_{X, \pm}, T^*_{X, -}\boxtimes \wedge^{j^\p-1}T^*_{B, -}\big)$. 

\begin{LEM}\label{bcurfgfg9f893hf8h30f30}
With $(\Mcl^\p, \Mcl, \Mcl^{\p\p})$; $(\Nc^\p, \Nc, \Nc^{\p\p})$ and $\Bc$ so chosen, there will exist a morphism of exact sequences as in \eqref{rfh94hf89hf80f093}.
\end{LEM}

\noindent
Assuming Lemma \ref{bcurfgfg9f893hf8h30f30} we can now make sense of the equality in \eqref{fjbjvyueviubonoinioeipe}. The proof of the present Proposition \ref{fjbbvubiuvbuenoienienp} now follows from \eqref{fjbjfbvybuveoioeeee}.
\qed
\\\\

\noindent
We will give a proof of Lemma \ref{bcurfgfg9f893hf8h30f30} below in order to complete the proof of Proposition \ref{fjbbvubiuvbuenoienienp}.
\\

\noindent
\emph{Proof of Lemma \ref{bcurfgfg9f893hf8h30f30}}.
We need to construct the morphism of short exact in \eqref{rfh94hf89hf80f093}. Observe that we immediately have the following maps:
\begin{align}
\begin{array}{ll}
\mathcal Hom_{\Oc_Y} (\Mcl^{\p\p}, \Mcl^\p) \otimes \Bc \lra \Nc^\p& \mbox{composition}
\\
\mathcal Hom_{\Oc_Y} (\Mcl^{\p\p}, \Mcl^{\p\p}) \otimes \Bc \lra \Nc^{\p\p} & \mbox{composition and wedge product}
\end{array}
\label{rhf84hf89fh83hf3h0fh3}
\end{align}
It remains to define a morphism $\mathcal Hom_{\Oc_Y}(\Mcl^{\p\p}, \Mcl) \otimes \Bc \ra \Nc$. Here we will use the fact that $\Nc = F^{~j^\p}_{j^\p-1}$ is the $(j^\p-1)$-th filtered piece of $\wedge^{j^\p}T^*_{Y, -}$. It can be explicitly described as follows. With $T^*_{Y,-}$ defined by the extension $T^*_{B, -}\ra T^*_{Y, -}\ra T^*_{X, -}$, the inclusion $\wedge^{j-1}T^*_{B, -}\subset \wedge^{j^\p-1}T^*_{Y, -}$ and the wedge product $\wedge^{j^\p-1}T^*_{Y, -}\otimes T^*_{Y, -}
\ra
\wedge^{j^\p}T^*_{Y, -}$ gives the following map
\[
\wedge^{j^\p-1}T^*_{B, -} \otimes T^*_{Y, -}\lra \wedge^{j^\p-1}T^*_{Y, -}\otimes T^*_{Y, -}
\lra
\wedge^{j^\p}T^*_{Y, -}.
\]
We can identify $F^{~j^\p}_{j^\p-1} = \img\{\wedge^{j^\p-1}T^*_{B, -} \otimes T^*_{Y, -}\lra \wedge^{j^\p}T^*_{Y, -}\}$. Now with the wedge product $T^*_{B, -}\otimes \wedge^{j^\p-1}T^*_{B, -}\ra \wedge^{j^\p}T^*_{B, -}$ we obtain a morphism of exact sequences,
\[
\xymatrix{
0\ar[r] & \ar[d]T^*_{B, -}\otimes \wedge^{j^\p-1}T^*_{B, -}\ar[r] & \ar[d] T^*_{Y, -}\otimes \wedge^{j^\p-1}T^*_{B, -}\ar[r] &\ar@{=}[d] T^*_{X, -}\otimes \wedge^{j^\p-1}T^*_{B, -}\ar[r] & 0
\\
0 \ar[r] & \wedge^{j^\p}T^*_{B, -} \ar[r] & F^{~j^\p}_{j^\p-1}\ar[r] & T^*_{X, -}\otimes \wedge^{j^\p-1}T^*_{B, -}\ar[r] & 0
}
\]
As can be checked, the map $T^*_{Y, -}\otimes \wedge^{j^\p-1}T^*_{B, -}\ra F^{~j^\p}_{j^\p-1}$ above now gives the desired map $\mathcal Hom_{\Oc_Y}(\Mcl^{\p\p}, \Mcl)\otimes \Bc\ra \Nc$ and, with the maps in \eqref{rhf84hf89fh83hf3h0fh3}, we obtain the desired morphism of exact sequences.
\qed

\newpage

\addtocontents{toc}{\vspace{\normalbaselineskip}}
\bibliographystyle{alpha}
\bibliography{Bibliography}

\begin{thebibliography}{Bet19b}

\bibitem[Ber87]{BER}
F.~A. Berezin.
\newblock {\em Introduction to Superanalysis}.
\newblock D. Reidel Publishing Company, 1987.

\bibitem[Bet16]{BETTPHD}
K.~Bettadapura.
\newblock {\em Obstruction Theory for Supermanifolds and Deformations of
  Superconformal Structures}.
\newblock PhD thesis, The Australian National University,
  \href{http://hdl.handle.net/1885/110239}{hdl.handle.net/1885/110239},
  December 2016.

\bibitem[Bet18a]{BETTEMB}
K.~Bettadapura.
\newblock Embeddings of complex supermanifolds.
\newblock Available at:
  \href{http://arxiv.org/abs/1806.02763}{arXiv:1806.02763} [math.AG], 2018.

\bibitem[Bet18b]{BETTHIGHOBS}
K.~Bettadapura.
\newblock Higher obstructions of complex supermanifolds.
\newblock {\em SIGMA}, 14(094), 2018.

\bibitem[Bet18c]{BETTSRS}
K.~Bettadapura.
\newblock On the problem of splitting deformations of super {Riemann} surfaces.
\newblock {\em Lett. Math. Phys.}, 2018.

\bibitem[Bet19a]{BETTOBSTHICK}
K.~Bettadapura.
\newblock Obstructed {Thickenings} and {Supermanifolds}.
\newblock {\em J. Geom. and Phys.}, 139:25--49, 2019.

\bibitem[Bet19b]{BETTAQ}
K.~Bettadapura.
\newblock Sheaves of {AQ} normal series and supermanifolds.
\newblock Available at
  \href{https://arxiv.org/abs/1903.03892}{arXiv:1903.03892} [math.AG], 2019.

\bibitem[Bre97]{BREDSHEAF}
G.~Bredon.
\newblock {\em Sheaf Theory (Second Edition)}.
\newblock Springer-Verlag, 1997.

\bibitem[Bry08]{BRY}
J.~Brylinski.
\newblock {\em Loop Spaces, Characteristic Classes and Geometric Quantization}.
\newblock Modern Birkhaeuser Classics, 2008.

\bibitem[DM99]{QFAS}
P.~Deligne and J.~W. Morgan.
\newblock {\em Quantum Fields and Strings: A course for Mathematicians},
  volume~1, chapter Notes on Supersymmetry (following Joseph Bernstein), pages
  41--97.
\newblock American Mathematical Society, Providence, 1999.

\bibitem[DW14]{DW2}
R.~Donagi and E.~Witten.
\newblock Super {Atiyah} classes and obstructions to splitting of supermoduli
  space.
\newblock available at: \href{http://arxiv.org/abs/1404.6257}{arXiv:1404.6257}
  [hep-th], 2014.

\bibitem[DW15]{DW1}
R.~Donagi and E.~Witten.
\newblock Supermoduli space is not projected.
\newblock In {\em Proc. Symp. Pure Math.}, volume~90, pages 19--72, 2015.

\bibitem[GR79]{GRAUST}
H.~Grauert and R.~Remmert.
\newblock {\em Theory of Stein Spaces}.
\newblock Springer-Verlag, Berlin, 1979.

\bibitem[Gre82]{GREEN}
P.~Green.
\newblock On holomorphic graded manifolds.
\newblock {\em Proc. Amer. Math. Soc.}, 85(4):587--590, 1982.

\bibitem[Huy05]{HUYB}
D.~Huybrechts.
\newblock {\em Complex Geometry: An Introduction}.
\newblock Springer-Verlag, Berlin, 2005.

\bibitem[Kod86]{KS}
K.~Kodaira.
\newblock {\em Complex Manifolds and Deformation of Complex Structures}.
\newblock Springer, 1986.

\bibitem[Man88]{YMAN}
Y.~Manin.
\newblock {\em Gauge Fields and Complex Geometry}.
\newblock Springer-Verlag, 1988.

\bibitem[Oni99]{ONISHCLASS}
A.~L. Onishchik.
\newblock On the classification of complex analytic supermanifolds.
\newblock {\em Lobachevskii J. Math.}, pages 47--70, 1999.

\bibitem[Rot85]{ROTHDEF}
M.~Rothstein.
\newblock Deformations of complex supermanifolds.
\newblock {\em Proc. Amer. Math. Soc.}, 95(2):255--60, October 1985.

\bibitem[Vai90]{VAIN}
Y.~Vaintrob.
\newblock Deformation of complex superspaces and coherent sheaves on them.
\newblock {\em J. Soviet Math.}, 51(1):2140--2188, August 1990.

\bibitem[Vis04]{VIST}
A.~Vistoli.
\newblock Notes on {Grothendieck} topologies, fibered categories and descent
  theory.
\newblock available at:
  \href{https://arxiv.org/abs/math/0412512}{arXiv:math/0412512} [math.AG],
  2004.

\end{thebibliography}

\hfill
\\
\noindent
\small
\textsc{
Kowshik Bettadapura 
\\
\emph{Yau Mathematical Sciences Center} 
\\
Tsinghua University
\\
Beijing, 100084, China}
\\
\emph{E-mail address:} \href{mailto:kowshik@mail.tsinghua.edu.cn}{kowshik@mail.tsinghua.edu.cn}

\end{document}